\documentclass[reqno]{amsart}
\usepackage{amsmath,amsthm,amsfonts, amssymb,amscd}
\usepackage[all]{xy}

\numberwithin{equation}{section}

\newtheorem{theorem}{Theorem}[section]
\newtheorem{lemma}[theorem]{Lemma}

\theoremstyle{definition}
\newtheorem{definition}[theorem]{Definition}

\newtheorem{remark}[theorem]{Remark}

\newcommand{\lex}{\,\overrightarrow{\times}\,}

\newcommand{\Rad}{\mbox{\rm Rad}}

\begin{document}
\title[$n$-dimensional Observables on $k$-Perfect MV-Algebras and $k$-Perfect Effect Algebras. II]{$n$-dimensional Observables on $k$-Perfect MV-Algebras and $k$-Perfect Effect Algebras. II. One-to-one Correspondence}
\author[A. Dvure\v{c}enskij, D. Lachman]{Anatolij Dvure\v{c}enskij$^{1,2}$, Dominik Lachman$^2$}
\maketitle

\begin{center}  \footnote{


The first author acknowledges the support by
the Slovak Research and Development Agency under contract APVV-16-0073 and the grant VEGA No. 2/0142/20 SAV, and the second author acknowledges the support by the Austrian Science Fund (FWF): project I 4579-N and the Czech Science Foundation (GA\v CR): project 20-09869L.}
Mathematical Institute,  Slovak Academy of Sciences\\
\v Stef\'anikova 49, SK-814 73 Bratislava, Slovakia\\
$^2$ Depart. Algebra  Geom.,  Palack\'{y} Univer.\\
17. listopadu 12, CZ-771 46 Olomouc, Czech Republic\\

E-mail: {\tt
dvurecen@mat.savba.sk,\quad dominiklachman@seznam.cz}
\date{}%
\end{center}

\begin{abstract}
The paper is a continuation of the research on a one-to-one correspondence between $n$-dimensional spectral resolutions and $n$-dimensional observables on lexicographic types of quantum structures  which started in \cite{DvLa4}.
In Part I, we presented the main properties of $n$-dimensional spectral resolutions and observables, and we deeply studied characteristic points which are crucial for our study.
In present Part II, there is a main body of our research. We investigate a one-to-one correspondence between $n$-dimensional observables and $n$-dimensional spectral resolutions with values in a kind of a lexicographic form of quantum structures like perfect MV-algebras or perfect effect algebras. The multidimensional version of this problem is more complicated than a one-dimensional one because if our algebraic structure is $k$-perfect for $k>1$, then even for the two-dimensional case of spectral resolutions we have more characteristic points. The obtained results are applied to existence of an $n$-dimensional meet joint observable of $n$ one-dimensional observables on a perfect MV-algebra and a sum of $n$-dimensional observables.
\end{abstract}

\section*{Introduction}
In this paper we continue the research from \cite{DvLa4}, where we have introduced $n$-dimensional observables and $n$-dimensional spectral resolutions defined on $k$-perfect effect algebras and $k$-perfect MV-algebras. In the first part, we have presented basic properties of observables and of spectral resolutions and we have concentrated to characteristic points of $n$-dimensional spectral resolutions which are important for the study of spectral resolutions. We underline that characteristic points are appearing only when the algebraic structure is of a lexicographic form. We note that an $n$-dimensional spectral resolution is a mapping $F:\mathbb R^n \to \Gamma_{ea}(H \lex G,(u,0))$, where $(H,u)$ is a unital po-group with interpolation and $G$ is a directed Dedekind $\sigma$-complete po-group with interpolation.

In the second part, we present the main results of our research which deal with establishing a one-to-one correspondence between $n$-dimensional spectral resolutions and $n$-dimensional observables showing that every $n$-dimensional spectral resolution can be uniquely extended to a unique $n$-dimensional observable. In addition, we apply the results to show existence of a kind of a joint $n$-dimensional observable of $n$ one-dimensional observables and to show how we can define a sum of $n$-dimensional observables.

Sections, theorems, propositions, lemmas, examples, and equations are numbered in continuation of \cite{DvLa4}, where there are basic notions used also in this part.

\section{General Spectral Resolutions and Two-Dimensional Ones }

In the section, we present a strengthened definition of an $n$-dimensional spectral resolution and of an $n$-dimensional pseudo spectral resolution valid for both MV-algebras as well as for effect algebras. This new definition will be sufficient in the next sections in order to extend an $n$-dimensional spectral resolution on a perfect MV-algebra to an $n$-dimensional observable $x$ such that $F=F_x$. First, the extension will be done for $n=2$ and then for general $n\ge 1$ for perfect MV-algebras and perfect effect algebras, and later also for $k$-perfect ones.

It is important to note that if $M$ is a $\sigma$-complete MV-algebra or a Dedekind monotone $\sigma$-complete effect algebra with (RDP), to show a one-to-one relationship between $n$-dimensional spectral resolutions and $n$-dimensional observables, the definition of a spectral resolution using (3.3)--(3.7) was sufficient, see \cite{DvKu,DvLa2,DvLa3}. However, if $M$ is perfect or $k$-perfect, that was not sufficient, see \cite{DDL,DvLa, DvLa1}. Therefore, the definition of an $n$-dimensional spectral resolution for $k$-perfect MV-algebras has to be strengthened and in this paper we need the following definition. The same is true for effect algebras. We note that every $F=F_x$, where $x$ is an $n$-dimensional observable does satisfy this new definition.

\begin{definition}\label{de:n-dim}
A mapping $F:\mathbb R^n \to M$ is said to be an {\it $n$-dimensional spectral resolution} if the following conditions hold
\begin{itemize}
\item[{\rm (i)}] the  volume condition,
\item[{\rm (ii)}] $$\bigvee_{(s_1,\ldots,s_n)\ll (t_1,\ldots,t_n)}F(s_1,\ldots,s_n)=F(t_1,\ldots,t_n),$$
\item[{\rm (iii)}] $\bigwedge_{t_i} F(s_1,\ldots,s_{i-1},t_i,s_{i+1},\ldots,s_n)=0$ for each $i=1,\ldots,n$,
\item[{\rm (iv)}] $\bigvee_{(s_1,\ldots,s_n)}F(s_1,\ldots,s_n)=1$,
\item[{\rm (v)}] If $(t^B_1,\ldots,t^B_n)$ is a characteristic point of $F$ corresponding to a block $B$, then the element
    \begin{equation}\label{eq:aB}
    a_B=\bigwedge\{F(s_1,\ldots,s_n)\colon (s_1,\ldots,s_n)\gg (t^B_1,\ldots,t^B_n)\}
    \end{equation}
    exists in $M$ (and it belongs to $B$).
\end{itemize}
If a mapping $F:\mathbb R^n\to M$ satisfies {\rm (i)--(iv)}, where instead of {\rm (iv)} we have
\begin{itemize}
\item[{\rm (iv)'}] $\bigvee_{(s_1,\ldots,s_n)} F(s_1,\ldots,s_n)=u_0$ and $u_0$ is not necessarily $1$,
    $F$ is said to be an {\it $n$-dimensional pseudo spectral resolution}.
\end{itemize}
\end{definition}

It is necessary to notify that according to Propositions 3.3--3.4, property (v) of Definition \ref{de:n-dim} is a necessary condition for existence of an $n$-dimensional observable $x$ such that $F=F_x$.

We note that for an $n$-dimensional pseudo spectral resolution, it can happen that $F$ has no characteristic point. Indeed, this can happen if $M$ is a perfect MV-algebra and $u_0\in \Rad(M)$. For example, if $F$ is an $n+1$-dimensional spectral resolution on a $\Gamma(\mathbb Z \lex G,(1,0))$, $n>1$. Given fixed $t \in \mathbb R$, the mapping $F_t:\mathbb R^n \to M$ defined by $F_t(t_1,\ldots,t_n):=F(t_1,\ldots,t_n,t)$, $t_1,\ldots,t_n \in \mathbb R$, is an $n$-dimensional pseudo spectral resolution, for more info see Lemma \ref{le:obser}. If $u_0=F(\infty,\ldots,\infty,t)\in \Rad(M)$, then $F_t$ has no characteristic point. This can happen always $t\le t^0_{n+1}$, where $(t^0_1,\ldots,t^0_{n+1})$ is a unique characteristic point of $F$.

It is important to notify that the characteristic points of $n$-dimensional pseudo spectral resolutions are evaluated in the same way as ones for $F$ with (i)--(iv) if they exist.

Moreover, if $u_0\in \Rad(M)$, $M=\Gamma(\mathbb Z\times \mathbb Z,(1,0))$, then the $n$-dimensional pseudo spectral resolution $F$ is in fact an $n$-dimensional spectral resolution on the interval $\sigma$-complete MV-algebra (monotone $\sigma$-complete effect algebra) $[0,u_0]\subset \Rad(M)$. If $u_0\in \Rad(M)'$, then $F$ is an $n$-dimensional spectral resolution on the interval algebra $[0,u_0]=\Gamma(\mathbb Z \lex \mathbb Z,(1,-g_0))$, where $(1,-g_0)=u_0$, $g_0\in G^-$, satisfying (i)--(v) and it is ``almost close" to a perfect MV-algebra, with $\Rad([0,u_0])=\Rad(M)=\{(0,g)\colon g\in G^+\}$ which is a Dedekind $\sigma$-complete poset. The ``almost close" means that $F$ has a (unique) characteristic point on $[0,u_0]$. Then our task is to find an $n$-dimensional observable $x$ on $[0,u_0]$ such that $x((-\infty,t_1)\times \cdots \times (-\infty,t_n))=F(t_1,\ldots,t_n)$, $t_1,\ldots, t_n \in \mathbb R$.

A one-to-one correspondence between $n$-dimensional spectral resolutions and $n$-dimensional observables for $\sigma$-complete MV-algebras and for Dedekind $\sigma$-complete effect algebras with the Riesz Decomposition property was established in \cite[Thm 5.1, Thm 5.2]{DvLa3}. Therefore, in what follows, we concentrate to $n$-dimensional spectral resolutions on lexicographic MV-algebras $\Gamma(H\lex G,(u,0))$ or on lexicographic effect algebras $\Gamma_{ea}(H\lex G,(u,0))$, where in the first case $(H,u)$ is a unital linearly ordered group and $G$ is a Dedekind $\sigma$-complete $\ell$-group, and in the second case $(H,u)$ is a unital po-group with interpolation and $G$ is a directed Dedekind monotone $\sigma$-complete po-group with interpolation.

If $M=\Gamma(G,u)$, where $(G,u)$ is a Dedekind $\sigma$-complete $\ell$-group, then (iv) and (v) are superfluous and $F$ can be extended to an $n$-dimensional observable, see \cite{DvLa2} for $n=2$ and \cite{DvLa3} for any $n\ge 1$. The same holds for a perfect effect algebra $\Gamma_{ea}(\mathbb Z\lex G,(1,0))$, where $G$ is a directed Dedekind monotone $\sigma$-complete po-group with interpolation, and for $\Gamma_{ea}(G,u)$, where $(G,u)$ is a monotone Dedekind $\sigma$-complete unital po-group with interpolation, (iv) and (v) are also superfluous.

In what follows, we will use the lexicographic product $\mathbb Z \lex G$, where $G$ is an Abelian directed po-group with interpolation or $G$ is an Abelian $\ell$-group. Then due to \cite[Cor 2.12]{Goo}, $\mathbb Z \lex G$ is with interpolation. Consequently, the effect algebra $\Gamma_{ea}(\mathbb Z\lex G,(n,-g_0))$, where $g_0\in G^+ $, is an effect algebra with (RDP), and $\Gamma_{ea}(\mathbb Z\lex G,(n,-g_0))$ is an MV-effect algebra if $G$ is an $\ell$-group.

We establish an important corollary of the definition of spectral resolutions holding for lexicographic effect algebras.

\begin{lemma}\label{le:infty}
Let $F$ be an $n$-dimensional spectral resolution with the finiteness property on an effect algebra $E=\Gamma_{ea}(H\lex G,(u,0))$, where $G$ is a directed Dedekind monotone $\sigma$-complete po-group with interpolation and $(H,u)$ is a unital po-group with interpolation.
If $i_1<\cdots < i_j$ is any non-empty subset of $\{1,\ldots,n\}$, then the element $\bigvee_{s_{i_1},\ldots,s_{i_j}}F(s_1,\ldots,s_n)$ exists in $M$ and we denote
$$
F(\hat s_1,\ldots,\hat s_n):=\bigvee_{s_{i_1},\ldots,s_{i_j}}F(s_1,\ldots,s_n),
$$
where $\hat s_i = +\infty$ if $i=i_k$ for some $k=1,\ldots,j$ and $\hat s_i = s_i$ otherwise.
\end{lemma}

\begin{proof}
Due to \cite[Cor 2.12]{Goo}, the po-group $H\lex G$ has the interpolation property.
Let $\mathbf t_u =(t^u_1,\ldots,t^u_n)$ be a unique characteristic point of $T_u=\{(t_1,\ldots,t_n)\colon F(t_1,\ldots,t_n)\in E_u\}$. We will establish that
$$
\bigvee_{s_1,\ldots,s_j}F(s_1,\ldots,s_j,t_{j+1},\ldots,t_n),
$$
exists in $E$ for all fixed $t_{j+1},\ldots,t_n\in\mathbb{R}$.

There are two cases. Case (i): There are $s^0_{1},\ldots,s^0_j\in \mathbb R$ such that $F(s_1,\ldots,s_j,t_{j+1},\ldots,t_n)\in E_u:=\{(u,g)\colon g \in G^-\}$ for all $s_{1},\ldots s_j\in \mathbb R$ with $s_{1} >s^0_{1},\ldots,s_j>s^0_j$; then the statement trivially holds.

Case (ii): Case (i) does not hold. Since $F$ has the finiteness property, we can assume that there are $s^0_1,\ldots, s^0_j\in \mathbb R$ and $h \in [0,u)_H$
such that $(s_1,\ldots,s_j,t_{j+1},\ldots,t_n)\in T_h$ for all $s_1>s^0_1,\ldots,s_j>s^0_j$. Without loss of generality, we can assume that $s^0_1\ge t^u_1,\ldots,s^0_j\ge t^u_j$.

We denote by $i$ the number of $t_k$'s below $t^u_k$, $k=j+1,\ldots,n$. Then $i\ge 1$.

We prove the statement using induction on $i$ and $j$. Assume that $i=1$ and let $k=j+1$ be such that $t_k=t_{j+1}\le t^u_{j+1}$ and $F(s_1,\ldots,s_j,t_{j+1},\ldots,t_n)\in E_h$.

For every $\epsilon>0$, we have the following volume condition
\begin{align*}
F(s_1,\ldots,s_{j-1},s_j,t_{j+1},\ldots,t_n)-&F(s_1,\ldots,s_{j-1}, t^u_{j}+\epsilon,t_{j+1},\ldots,t_n)\\
\leq F(s_1,\ldots,s_{j-1},s_j,t^u_{j+1}+\epsilon,\ldots,t_n)- &F(s_1,\ldots,s_{j-1},t^u_{j}+\epsilon,t^u_{j+1}+\epsilon,\ldots,t_n).
\end{align*}
So
\begin{equation}\label{eq:b}
F(s_1,\ldots,s_{j-1},s_j,t_{j+1},\ldots,t_n)
\end{equation}
has an upper bound
\begin{align*}
F(s_1,\ldots,s_{j-1},t^u_{j}+\epsilon,t_{j+1},\ldots,t_n)&+ [F(s_1,\ldots,s_{j-1},s_j,t^u_{j+1}+\epsilon,\ldots,t_n)\\
&-F(s_1,\ldots,s_{j-1},t^u_{j}+\epsilon,t^u_{j+1}+\epsilon,\ldots,t_n)].
\end{align*}
Since we have finitely many blocks, we can
choose $h,h'\in [0,u]_H$, $\epsilon >0$, and sufficiently large variables $s_1,\ldots,s_j\in \mathbb R$ such that $F(s_1,\ldots,s_{j-1},s_j,t_{j+1},\ldots,t_n)\in E_{h}$, $F(s_1,\ldots,s_{j-1},t^u_{j}+\epsilon,t_{j+1},\ldots,t_n)\in E_{h}$, and $F(s_1,\ldots,s_{j-1},s_j,t^u_{j+1}+\epsilon,\ldots,t_n)$ and $F(s_1,\ldots,s_{j-1},t^u_{j}+\epsilon,t^u_{j+1}+\epsilon,\ldots,t_n)$ belong to $E_{h'}$.

Now, we follow the induction with respect to $i$ with fixed $j=1$. Then
\begin{equation}\label{eq:s1}
F(s_1,t_2,\ldots,t_n)\le F(t^u_1+\epsilon,t_2,\ldots,t_n) +
[F(s_1,t^u_2+\epsilon,t_3,\ldots,t_n)-F(t^u_1+\epsilon, t^u_2+\epsilon,t_3,\ldots,t_n)].
\end{equation}
Since $i=1$, then $h'=u$, $E_u\ni F(s_1,t^u_2+\epsilon,t_3,\ldots,t_n)\le 1$ and $E_0\ni F(s_1,t^u_2+\epsilon,t_3,\ldots,t_n)-F(t^u_1+\epsilon, t^u_2+\epsilon,t_3,\ldots,t_n) \le 1 - F(t^u_1+\epsilon, t^u_2+\epsilon,t_3,\ldots,t_n)\in E_0$ which shows that $\{F(s_1,t_2,\ldots,t_n)\colon s_1\in \mathbb R\}$ has an upper bound in $E_h$.

Now, assume that $i>1$. Due to the induction hypothesis, we assume that the lemma holds for $j=1$ and every $1\le i'<i$. From (\ref{eq:s1}) we conclude by induction that $\bigvee_{s_1}F(s_1,t^u_2+\epsilon,t_3,\ldots,t_n)$ exists in $E$ and it belongs to $E_{h'}$. Therefore, $\bigvee_{s_1}[F(s_1,t^u_2+\epsilon,t_3,\ldots,t_n)-F(t^u_1+\epsilon, t^u_2+\epsilon,t_3,\ldots,t_n)] \in E_0$ and
$\bigvee_{s_1} \{F(s_1,t_2,\ldots,t_n)\colon s_1\in \mathbb R\}$ has the supremum in $E_h$.

Consequently, $\bigvee_{s_1}F(s_1,t_2,\ldots,t_n)$ exists in $E_h$ for all $t_2,\ldots,t_n\in \mathbb R$.

Now, we assume that the lemma holds for each $j'$ with $1\le j'<j$. 
We use an upper bound for (\ref{eq:b}) that is just after (\ref{eq:b}). Due to the induction hypothesis, $$a_1:=\bigvee\{F(s_1,\ldots,s_{j-1},t^u_{j}+ \epsilon,t_{j+1},\ldots,t_n)\colon s_1,\ldots,s_{j-1}\in \mathbb R\}
$$
exists in $E$ and it belongs to $E_h$. Since $F(s_1,\ldots,s_{j-1},t^u_{j}+\epsilon,t^u_{j+1}+\epsilon,\ldots,t_n)$ belongs to $E_{h'}$ for sufficiently large $s_1,\ldots,s_{j-1}$, it belongs to some block $B\subseteq E_{h'}$. Using definition of spectral resolutions, the element $a_B=\bigwedge\{F(u_1,\ldots,u_n)\colon (u_1,\ldots,u_n)\in B\}$ exists in $E$ and it belongs to $E_{h'}$. Finally, as in the previous case for $j=1$, we can show $\{F(s_1,\ldots,s_{j-1},s_j,t^u_{j+1}+\epsilon,\ldots,t_n)\colon s_1,\ldots,s_j\in \mathbb R\}$ has an upper bound, say $a_2$, in $E_{h'}$: First we show it for $i=1$ and then for $i>1$. Altogether,
$\{F(s_1,\ldots,s_j,t_{j+1},\ldots,t_n)\colon s_1,\ldots,s_j\in \mathbb R\}$ has an upper bound in $E_h$, namely $a_1+a_2-a_B$, which finishes the proof.
\end{proof}

If $F$ is two-dimensional, then, in particular, we have $F(s,t_0^+):=\bigwedge_{t>t_0}F(s,t)$ exists in $M$ and $F(s,t_0^+)\in \Rad(M)$ if $s<s_0$ and $F(s,t_0^+)\in \Rad(M)'$ if $s> s_0$. Dually, $F(s_0^+,t)\in \Rad(M)$ if $t\le t_0$ and, $F(s_0^+,t)\in \Rad(M)'$ if $t\ge t_0$.

\begin{lemma}\label{lem:help1}
Let $G$ be a directed monotone $\sigma$-complete unital po-group with interpolation, fix $g_0\in G^+$, and let $(a_i)_i,(b_i)_i,(c_i)_i,(d_i)_i$ be sequences from  $E=\Gamma_{ea}(\mathbb Z \lex G,(n,-g_0))$, $n\ge 1$.

{\rm (1)}  If $(a_i)_i,(b_i)_i,(c_i)_i,(d_i)_i$ are non-decreasing with suprema $a$, $b$, $c$, $d$, such that $(a_i+d_i-b_i-c_i)_i$ is non-decreasing and from $E$. Then $\bigvee_i(a_i+d_i-b_i-c_i)$ and $a+b-b-c$ exist in $E$ and $\bigvee_i(a_i+d_i-b_i-c_i)=a+d-b-c$.

{\rm (2)} If $(a_i)_i,(b_i)_i,(c_i)_i,(d_i)_i$ are non-increasing with infima $a$, $b$, $c$, $d$, such that $(a_i+d_i-b_i-c_i)_i$ is non-increasing and from $E$. Then $\bigwedge_i(a_i+d_i-b_i-c_i)$ and $a+b-b-c$ exist in $E$ and $\bigwedge_i(a_i+d_i-b_i-c_i)=a+d-b-c$.

{\rm (3) } If $(a_i)_i,(b_i)_i,(c_i)_i,(d_i)_i$ are non-decreasing with suprema $a$, $b$, $c$, $d$, such that $(a_i+d_i-b_i-c_i)_i$ is non-increasing and from $E$. Then $\bigwedge_i(a_i+d_i-b_i-c_i)$ and $a+b-b-c$ exist in $E$ and $\bigwedge_i(a_i+d_i-b_i-c_i)=a+d-b-c$.

{\rm (4) } If $(a_i)_i,(b_i)_i,(c_i)_i,(d_i)_i$ are non-increasing with infima $a$, $b$, $c$, $d$, such that $(a_i+d_i-b_i-c_i)_i$ is non-decreasing and from $E$. Then $\bigvee_i(a_i+d_i-b_i-c_i)$ and $a+b-b-c$ exist in $E$ and $\bigvee_i(a_i+d_i-b_i-c_i)=a+d-b-c$.

In particular, these hold when $c_i,d_i$ equal constantly zero.

Moreover, the statements hold also for $E=\Gamma_{ea}(G,u)$, if $u$ is a strong unit for $G$.
\end{lemma}

\begin{proof}
First we note that a sequence $(e_i)_i$ of elements of $E= \Gamma_{ea}(\mathbb Z \lex G,(n,-g_0))$ has a supremum (infimum) in $E$ iff there are $k=0,\ldots,n$, $i_0\ge 1$, and $a,b\in G$ such that $(k,a)\le e_i=(k,g_i)\le (k,b)$ for each $i\ge i_0$. Moreover, the supremum/infimum of $(e_i)_i$ taken in $E$ and in $\mathbb Z\lex G$ exists simultaneously and then they are the same.

Therefore, we can without loss of generality assume the four sequences are constant in the first component, and hence so is the fifth one. More concretely, there is some $k=1,\ldots,n-1$ such that $a_i+d_i-b_i-c_i$ belong to $\{k\}\times G$ (respectively $\{0\}\times G^+$ or $\{n\}\times \{-g\colon g \ge g_0, g \in G^+\}$ in the case $k=0$ or $k=n$) for each $i\ge 1$. This sequence is monotone and bounded in $\{k\}\times G$ (respectively $\{0\}\times G^+$ or $\{n\}\times G^-$): In the case of (1) it has an upper bound $a+d-b_1-c_1$ (respective $(n,-g_0)$ in the case $k=n$) and similarly one can find an upper bound in the case (3) and a lower bound in the cases (2) and (4). By the assumptions on $G$, we see, that $S=\bigvee_i(a_i+d_i-b_i-c_i)$ (respectively $S=\bigwedge_i(a_i+d_i-b_i-c_i)$) exists in $E$.

It remains to prove, for each of the four cases, the desired equality (the existence of $a+d-b-c$ in $E$ is then a consequence, as it exists in the po-group $\mathbb{Z}\lex G$).

(1) Denote $S=\bigvee_i(a_i+d_i-b_i-c_i)$, then for each $i$, we have $a_i+d_i\leq S+ b_i+c_i$, hence $a+d\leq S+ b+c$ which gives us one inequality. To prove the second one, we have to verify that for each $i$ we have the inequality
$$
a_i+d_i- b_i-c_i\leq a+d-b-c.
$$
Equivalently, $a_i+d_i+b+c\leq b_i +c_i +a+d$. But due to monotonicity of $(a_i+b_i-c_i-d_i)_i$, for each $j\geq i$, we have $a_i+d_i+b_j+c_j\leq b_i +c_i +a_j+d_j$.

(2) The proof is dual to (1).

(3) The proof is similar to (1). Denote $S=\bigwedge_i(a_i+d_i-b_i-c_i)$. We have $a_i+d_i\geq S+ b_i+c_i$ for each $i$, which gives us $a+d-b-c\geq S$. And for each $i$, $a-b-c+d\leq a_i-b_i-c_i+d_i$, as this is equivalent to $a+d+b_i+c_i\leq b+c+a_i+d_i$, which follows from: $$\forall j>i, a_j+d_j+b_i+c_i\leq b_j+c_j+a_i+d_i.$$

The last case (4) is dual to (3).
\end{proof}

To lighten the notation, we will write $F(s^+,t)$ in place of $\bigwedge_{s'>s}F(s',t)$ and $F(s^-,t)$ in place of $\bigvee_{s<s'}F(s',t)$ and similarly for the second coordinate.

\begin{lemma}\label{lem:help2}
For each $t_1\in\mathbb{R}\cup\{-\infty\}$, $s_1\in\mathbb{R}\cup\{+\infty\}$, and for each two-dimensional spectral resolution $F(s,t)$, we have
\begin{equation}\label{eq:A}
\bigvee_{s<s_1}\bigwedge_{t>t_1}F(s,t)=\bigwedge_{t>t_1}\bigvee_{s<s_1}F(s,t),
\end{equation}
provided the infima exist. Consequently $F(s^-,t^+)$ is well defined.
\end{lemma}

\begin{proof}
Let $s_1,t_1\in \mathbb R$. We are going to prove that the left hand side of (\ref{eq:A}) equals $\bigwedge_{t>t_1}F(s_1,t)$. We can deduce from the volume condition
$$
0\leq \bigwedge_{t>t_1} F(s_1,t)-\bigwedge_{t>t_1}F(s,t)\leq F(s_1,t')-F(s,t')
$$
for each $t'>t_1$ and each $s<s_1$. But $F(s_1,t')-F(s,t')\searrow 0$ as $s\nearrow s_1$, hence
$$
0=\bigwedge_{s<s_1}[\bigwedge_{t>t_1} F(s_1,t)-\bigwedge_{t>t_1}F(s,t)]=\bigwedge_{t>t_1} F(s_1,t)-\bigvee_{s<s_1}\bigwedge_{t>t_1}F(s,t),
$$
which is the desired equation. The case $s_1=+\infty$ proceed in almost the same way and the case $t_1=-\infty$ is trivial.
\end{proof}

The last two lemmas will be frequently used.

Now, we present the first basic result on a one-to-one correspondence between spectral resolutions and observables for $n=2$.

\begin{theorem}\label{th:TwoMV}
Let $M=\Gamma(\mathbb Z\lex G,(1,0))$ be a perfect MV-algebra, where $G$ is a Dedekind $\sigma$-complete $\ell$-group. If $F$ is a two-dimensional spectral resolution on $M$, then there is a unique two-dimensional observable $x$ on $M$ such that $x((-\infty,s)\times (-\infty,t))=F(s,t)$ for each $(s,t)\in \mathbb R^2$.
\end{theorem}

\begin{proof}
Denote by $(s_0,t_0)\in\mathbb{R}^2$ the characteristic point of $F$. We divide the plane $\mathbb{R}^2$ into four blocks by cutting $\mathbb R^2$ in each coordinate of the characteristic point. Denote $B_{0,0}=\{(s,t)\colon (s,t)\leq (s_0,t_0)\}$, $B_{0,1}=\{(s,t)\colon s\leq s_0, t_0<t\}$, $B_{1,0}=\{(s,t)\colon s_0< s, t\leq t_0\}$ and $B_{1,1}=\{(s,t)\colon (s_0,t_0)\ll (s,t)\}$. For each block we define a spectral resolution which encodes what is the essential increase on the block. The case of $B_{0,0}$ is trivial: Define
$$
F_{0,0}(s,t):=F(s',t'), \mbox{where } s'=\min\{s_0,s\} \mbox{ and } t'=\min\{t_0,t\}.
$$
In the case of $B_{0,1}$, we define
$$
F_{0,1}(s,t)= \left\{\begin{array}{ll}
F(s',t)-F(s',t_0^+) & \mbox{if}\ t>t_0, \mbox{ where } s'=\min\{s_0,s\},\\
0 & \mbox{if} \ t\le t_0.
\end{array}
\right.
$$
Similarly, we define
$$
F_{1,0}(s,t)= \left\{\begin{array}{ll}
F(s,t')-F(s_0^+,t') & \mbox{if}\ s>s_0, \mbox{ where } t'=\min\{t_0,t\},\\
0 & \mbox{if} \ s\le s_0.
\end{array}
\right.
$$

The most important is of course $F_{1,1}$ which is defined as follows
$$
F_{1,1}(s,t)=\left\{\begin{array}{ll}
F(s,t)-\bigwedge_{s'>s_0}F(s',t)-\bigwedge_{t'>t_0}F(s,t') +\bigwedge_{s'>s_0,t'>t_0}F(s',t'), & \mbox{if}\ s>s_0,t>t_0,\\
0, & \mbox{otherwise}.
\end{array}
\right.
$$
Using  the above notation, $F_{1,1}$ can be rewritten
$$
F_{1,1}(s,t)=F(s,t)- F(s_0^+,t)-F(s,t_0^+)+F(s_0^+,t_0^+), \quad (s,t)\in B_{1,1}.
$$
From the volume condition we conclude that $F_{1,1}(s,t)\ge 0$ and $F_{1,1}(s,t) \in \Rad(M)$ for each $s,t \in \mathbb R$: Indeed, let $(s,t)\in B_{1,1}$ and take $s',t'$ such that $s_0<s'<s$ and $t_0<t'<t$. Then the volume condition yields $F(s',t')+F(s,t)\ge F(s,t')+F(s',t)$, so that $F(s_0^+,t_0^+)+F(s,t)\ge F(s,t_0^+) +F(s_0^+,t)$.

Just defined $F_{i,j}$'s are two-dimensional pseudo spectral resolutions. Pseudo means that the top element for $F_{i,j}$, $u_{i,j}:=\bigvee_{s,t}F_{i,j}(s,t)$, is defined in $M$ and $F_{i,j}$ is in fact a two-dimensional observable in the interval $[0,u_{i,j}]$. Indeed, the first three $F_{i,j}$ are evidently so,
because due to the volume condition and Lemma \ref{lem:help1}, each of them is monotone in each variable.

The volume condition for $F_{1,1}$: Let $(s_1,t_1),(s_2,t_2)\in B_{1,1}$ with $(s_1,t_1)\le (s_2,t_2)$ be given, then the volume condition for $F$ gives
$$
F(s_2,t_2)+F(s_1,t_1) \ge F(s_1,t_2)+F(s_2,t_1).
$$
Adding to both sides $-F(s_2,t_0^+) - F(s_0^+,t_2)+F(s_0^+,t_0^+)- F(s_1,t_0^+)-F(s_0^+,t_1)+ F(s_0^+,t_0^+)$, we obtain $F_{1,1}(s_2,t_2)+F_{1,1}(s_1,t_1) \ge F_{1,1}(s_1,t_2)+F_{1,1}(s_2,t_1)$. The other possibilities of $(s_1,t_1), (s_2,t_2)\in \mathbb R^2$ are trivial.
In addition, $F_{1,1}(s_1,t)\le F_{1,1}(s_2,t)$ and $F_{1,1}(s,t_1)\le F_{1,1}(s,t_2)$ whenever $s_1\le s_2$, $t_1\le t_2$, $s,t \in \mathbb R$: Take $s_0< s'\le s_1<s_2$ and $t_2\ge t>t'>t_0$. Then from the volume conditions for $F$ on semi-closed rectangles $A_1=[s',s_1)\times [t',t)$ and $A_2=[s',s_2)\times [t',t)$, we have $V(F,A_1)\le V(F,A_2)$, which entails
\begin{align*}
F(s_2,t)+F(s',t')-F(s',t)-F(s_2,t')&\ge F(s_1,t)+F(s',t')-F(s',t)-F(s_1,t')\\
F(s_2,t)+F(s_0^+,t_0^+)-F(s_0^+,t)-F(s_2,t_0^+)&\ge F(s_1,t)+F(s_0^+,t_0^+)-F(s_0^+,t)-F(s_1,t_0^+)\\
F_{1,1}(s_2,t)\ge F_{1,1}(s_1,t).
\end{align*}
Similarly for the second coordinate.

The continuity $\bigwedge_s F_{1,1}(s,t)=0=\bigwedge_t F_{1,1}(s,t)$ is trivial due to definition of $F_{1,1}$ because, for each $s< s_0$ or $t< t_0$, we have $F_{1,1}(s,t)=0$.

The continuity $\bigvee_{(s_1,t_1)<(s,t)}F_{1,1}(s_1,t_1)=F_{1,1}(s,t)$ follows from Lemma~\ref{lem:help1}(1) and Lemma~\ref{lem:help2}.

Then all four mappings take values in the radical $\Rad(M)$.
Denote $u_{i,j}=\bigvee_{s,t}F_{i,j}(s,t)$. We have $u_{0,0}=F(s_0,t_0)$, $u_{0,1}= F(s_0,\infty)-F(s_0,t_0^+)$, $F_{1,0}= F(\infty,t_0)-F(s_0^+,t_0)$, and $u_{1,1}= u-F(s_0^+,\infty)-F(\infty,t_0^+)+F(s_0^+,t_0^+)= F(\infty,\infty)-F(s_0^+,\infty)-F(\infty,t_0^+)+F(s_0^+,t_0^+)$. Therefore, each $F_{i,j}$ is a two-dimensional spectral resolution on the $\sigma$-complete interval MV-algebra $[0,u_{i,j}]\subset \Rad(M)$, and so they could be extended to two-dimensional observables $x_{i,j}$, $i,j=0,1$, on $[0,u_{i,j}]$, see \cite{DvLa2}, such that $x_{i,j}((-\infty,s)\times (-\infty,t))=F_{i,j}(s,t)$, $s,t \in \mathbb R$. In addition, $x_{i,j}(A)=x_{i,j}(A\cap B_{i,j})$, $A \in \mathcal B(\mathbb R^2)$.

We would like to glue the observables, but obviously the four observables does not take care of what is happening on the characteristic point and rays going through the characteristic point. So we have to define, in addition, two one-dimensional spectral resolutions:
\begin{align*}
&F_0(s)=\bigwedge_{t>t_0}F(s,t)-F(s,t_0)= F(s,t_0^+)-F(s,t_0),\quad s\in \mathbb R,\\ &F_1(t)=\bigwedge_{s>s_0}F(s,t)-F(s_0,t)= F(s_0^+,t)-F(s_0,t),\quad t \in \mathbb R.
\end{align*}
\noindent
\vspace{3mm}
\textit{Claim: $F_0$ and $F_1$ are two-dimensional pseudo spectral resolutions.}
\vspace{2mm}

\begin{proof}
We prove the case of $F_0$. Monotonicity of $F_0$: For each $t>t_0$ and $s_1<s_2$, we have $F(s_1,t)-F(s_1,t_0)\leq F(s_2,t)-F(s_2,t_0)$ (by the volume condition). Denote $F(s,t_0^+):=\bigwedge_{t>t_0}F(s,t)$. Then $F(s,t_0^+)\in \Rad(M)$ if $s<s_0$, otherwise $F(s,t_0^+)\in \Rad(M)'$.
Clearly the mapping $s\mapsto F(s,t_0^+)$ is monotone and bounded, so there is $F(\infty,t_0^+)\in \Rad(M)'$. The same holds for $s\mapsto F(s,t_0)$: It is surely monotone and it is bounded as, for each real $s>s_0+1$, we have $F(s,t_0)-F(s_0+1,t_0)\leq F(s,t^+_0)-F(s_0+1,t^+_0)\leq F(\infty,t_0^+)-F(s_0+1,t^+_0)\in \mathrm{Rad}(M)$.
So we can define $F(\infty,t_0)$. Using Lemma~\ref{lem:help1},
we observe that $\bigvee_{s}(F(s,t^+_0)-F(s,t_0))=F(\infty,t^+_0)-F(\infty,t_0)$.
In the opposite case, we get $0\leq \bigwedge_{s}(F(s,t^+_0)-F(s,t_0))\leq F(-\infty,t^+_0)=0$.
An application of the special case in Lemma~\ref{lem:help1}(1) together with Lemma~\ref{lem:help2} gives us for each $t_1\in\mathbb{R}$:
\begin{align*}
\bigvee_{t<t_1}F_0(t)&=
\bigvee_{t<t_1}(F(s_0^+,t)-F(s_0,t))= \bigvee_{t<t_1}F(s_0^+,t)-\bigvee_{t<t_1}F(s_0,t)\\
&= \bigvee_{t<t_1}\bigwedge_{s>s_0}F(s,t)- \bigvee_{t<t_1}F(s_0,t)
=\bigwedge_{s>s_0} F(s,t_1)-F(s_0,t_1)=F_0(t_1).
\end{align*}
Now, we show that $\bigwedge_s F_0(t)=0$: Using monotonicity of $F_0$ and Lemma~\ref{lem:help1}(2), we have
$$
\bigwedge_s F_0(s) = \bigwedge_s(F(s,t_0^+)-F(s,t_0))= \bigwedge_s F(s,t_0^+) -\bigwedge_s F(s,t_0)\le \bigwedge_s(F(s,t_0^+ +1)- 0)=0-0=0.
$$
Then $F_0$ is a one-dimensional spectral resolution on a lexicographic MV-algebra $\Gamma(\mathbb Z\lex G,(1,-g_0))=[0,u_0]$, with $g_0\in G^+$ such that $(1,-g_0)=u_0$, where $u_0=F(\infty,t_0^+)-F(\infty,t_0)\in \Rad(M)'$.

Similarly, $F_1$ is a one-dimensional spectral resolution on a lexicographic MV-algebra $\Gamma(\mathbb Z\lex G,(1,-g_1))=[0,u_1]$, where $g_1\in G^+$ is such that $(1,-g_1)=u_1$ and $u_1=\bigvee_t F_0(t)= F(s_0^+,\infty)-F(s_0,\infty)\in \Rad(M)'$.
\end{proof}

Now, we continue in the proof of Theorem \ref{th:TwoMV}. The result \cite[Thm 4.8]{DDL}, which holds also for our lexicographic MV-algebras, gives us an existence of one-dimensional observables $x_0$ and $x_1$ on $[0,u_0]$ and $[0,u_1]$, respectively, extending the spectral resolutions $F_0$ and $F_1$.

Finally, we can glue the above observables to define an observable $x$ on $M$ by the prescription:
\begin{equation}\label{eq:obser'}
x(A)=\sum_{i,j=0,1}x_{i,j}(A)+\sum_{i=0}^1 x_i(\pi_i(A\setminus\{(s_0,t_0)\}))+\chi_{(s_0,t_0)}(A)\cdot x_0(\{s_0\}), \quad A \in \mathcal B(\mathbb R^2),
\end{equation}
or, equivalently,
\begin{equation}\label{eq:obser2}
x(A)=\sum_{i,j=0,1}x_{i,j}(A)+\sum_{i=0}^1 x_i(\pi_i(A))-\chi_{(s_0,t_0)}(A)\cdot x_0(\{s_0\}), \quad A \in \mathcal B(\mathbb R^2),
\end{equation}
where $\pi_0$ and $\pi_1$ are projections from  $\mathbb R^2$ onto $\mathbb R$ such that $\pi_0(s,t)=s$, $\pi_1(s,t)=t$, $(s,t)\in \mathbb R^2$.
One can verify
\begin{align*}
x_0(\{s_0\})
&=\bigwedge_{s>s_0}(\bigwedge_{t>t_0}F(s,t)-F(s,t_0)) -(\bigwedge_{t>t_0}F(s_0,t)-F(s_0,t_0))\\
&=\bigwedge_{s>s_0,t>t_0} [(F(s,t)-F(s,t_0))-(F(s_0,t)+F(s_0,t_0))]\\
&= F(s_0^+,t_0^+) - F(s_0^+,t_0)- F(s_0,t_0^+)+ F(s_0,t_0)\\
&=x_1(\{t_0\}).
\end{align*}
Moreover, $x_0(\{s_0\})= x_1(\{t_0\})\in \Rad(M)'$.

We need to verify $x$ is really a two-dimensional observable on $M$.
See that (what should be the ``measure" of the upper half-plane)
\begin{align*}
x(B_{0,1}\cup B_{1,1})
&=x_{0,1}(\mathbb{R}^2)+x_1((t_0,\infty))+x_{1,1}(\mathbb{R}^2)\\&= [F(s_0,\infty)-F(s_0,t_0^+)]\\
&+[F(s_0^+,\infty)-F(s_0,\infty)-F(s_0^+,t_0^+)+F(s_0,t_0^+)]\\
&+[F(\infty,\infty)-F(\infty,t_0^+)-F(s_0^+,\infty)+F(s_0^+,t_0^+)]\\ &=F(\infty,\infty)-F(\infty,t_0^+).
\end{align*}
Similarly (the ``measure" of the lower half-plane), $x(B_{0,0}\cup B_{1,0})= x_{0,0}(\mathbb{R}^2)+x_1((-\infty,t_0))+x_{1,0}(\mathbb{R}^2) + x_0(\mathbb R) =F(\infty,t_0)+ x_0(\mathbb R)$.
Now putting all the pieces together we obtain $x(\mathbb{R}^2)=F(\infty,\infty)-F(\infty,t_0^+)+ F(\infty,t_0)+x_0(\mathbb{R})=F(\infty,\infty)=u$. Since $0\le x(A)\le x(\mathbb R^2)=u$, we see that $x(A)$ is correctly defined by (\ref{eq:obser2}) and it belongs to $M$ for every $A \in \mathcal B(\mathbb R^2)$. Moreover, (i) $x$ is monotone (see (\ref{eq:obser'})), (ii) if $A_1$ and $A_2$ are disjoint Borel sets in $\mathbb R^2$, then $x(A_1\cup A_2)=x(A_1)+x(A_2)$, and (iii) if $A=\bigcup_i A_i$, where $(A_i)_i$ is a sequence of non-decreasing Borel sets from $\mathcal B(\mathbb R^2)$, then $x(A)=\bigvee_i x(A_i)$.

It is tedious but straightforward to check the equality $F(s,t)=x((-\infty,s)\times(-\infty,t))$. For example, if $(s,t)\in B_{1,1}$, then
\begin{align*}
x((-\infty,s)\times(-\infty,t))&= F_{1,1}(s,t)+F_{0,1}(s_0,t)+ F_{0,0}(s_0,t_0)+ F_{1,0}(s,t_0)+F_0(s)+F_1(t)-x_0(\{s_0\})\\
&=[F(s,t)-F(s_0^+,t)-F(s,t_0^+) + F(s_0^+,t_0^+)]\\
&+[F(s_0,t)-F(s_0,t_0^+)] + F(s_0,t_0)+ [F(s,t_0)-F(s^+_0,t_0)]\\
&+ [F(s,t_0^+)- F(s,t_0)]+ [F(s_0^+,t)-F(s_0,t)]\\
&- [F(s_0^+,t_0^+)-F(s_0^+,t_0)- F(s_0,t_0^+)+ F(s_0,t_0)]\\
&=F(s,t).
\end{align*}

Uniqueness: Let $y$ be a two-dimensional observable on $M$ such that $F(s,t)=y((-\infty,s)\times (-\infty,t))$, $s,t \in \mathbb R^2$ and let $\mathcal K=\{A \in \mathcal B(\mathbb R^2)\colon y(A)=x(A)\}$. Then $\mathcal K$ contains $\mathbb R^2$, all intervals of the form $(-\infty,s)\times (-\infty,t)$, and is closed under complements and unions of disjoint sequences, i.e. $\mathcal K$ is a Dynkin system and by the Sierpi\'nski Theorem, \cite[Thm 1.1]{Kal}, $\mathcal K = \mathcal B(\mathbb R^2)$, i.e. $x=y$. Hence, $x$ is a unique two-dimensional observable on $M$ in question.
\end{proof}

The following second basic theorem for perfect effect algebras with (RDP) follows the same proof as that of Theorem \ref{th:TwoMV}. We notify that perfect effect algebras are categorically equivalent to Abelian directed po-group with interpolation, see \cite[Thm 5.7]{Dvu1}, and one-dimensional spectral resolutions on perfect effect algebras were studied in \cite{304,DvLa1}.

\begin{theorem}\label{th:TwoEA}
If $E=\Gamma_{ea}(\mathbb Z\lex G,(1,0))$ is a perfect effect algebra, where $G$ is a directed monotone $\sigma$-complete po-group with interpolation, then every two-dimensional spectral resolution can be extended to a unique two-dimensional observable on $E$.
\end{theorem}

\section{$n$-dimensional Spectral Resolutions}

Now, we deal with the general case of $n$-dimensional spectral resolutions for general lexicographic effect algebras $E=\Gamma_{ea}(H\lex G,(u,0))$.

The following lemma describes what does happen with monotone sequences appearing in the volume condition:

\begin{lemma}\label{lem:help3}
Suppose we have a non-decreasing sequence $(a_\delta^i)_i$ of elements of an effect algebra $E$, such that all the sequences are non-decreasing (non-increasing, respectively) for each $\delta\in\{0,1\}^n$. Moreover, suppose $(\sum_{\delta}(-1)^{\pi(\delta)}a_\delta^i)_i$ is non-decreasing or non-increasing. Then
$$
\bigvee_i\sum_{\delta}(-1)^{\pi(\delta)} a_\delta^i=\sum_\delta (-1)^{\pi(\delta)}\bigvee_i a^i_\delta,
(\text{respectively,\ }
\bigvee_i\sum_{\delta}(-1)^{\pi(\delta)} a_\delta^i=\sum_\delta (-1)^{\pi(\delta)}\bigwedge_i a^i_\delta),
$$
or
$$
\bigwedge_i\sum_{\delta}(-1)^{\pi(\delta)} a_\delta^i=\sum_\delta (-1)^{\pi(\delta)}\bigvee_i a^i_\delta,
(\text{respectively,\ }
\bigwedge_i\sum_{\delta}(-1)^{\pi(\delta)} a_\delta^i=\sum_\delta (-1)^{\pi(\delta)}\bigwedge_i a^i_\delta),
$$
where $\pi(\delta)=|\{d_i\colon \delta=(\delta_1,\ldots,\delta_n)\}|$ is the number of zero coordinates in $\delta=(\delta_1,\ldots,\delta_n)\in \{0,1\}^n$.

The appropriate version of the equation holds in each of the four combinations in the following sense: $(a_\delta^i)_i$'s are non-decreasing/non-increasing and  $(\sum_{\delta}\mathrm{sgn}(\delta)a_\delta^i)_i$ is non-decreasing/non-increasing.
\end{lemma}

\begin{proof}
The proof proceeds in a similar way as that for Lemma~\ref{lem:help1}, only more sequences are involved.
\end{proof}

\begin{lemma}\label{le:+,-}
The value $F(t_1^{\triangle_1},\ldots,t_n^{\triangle_n})$, where $\triangle_i\in\{+,-\}$ is well defined. That is, whenever we have an expression of the form $\Xi_1\cdots \Xi_n F(s_1,\ldots,s_n)$, where each $\Xi_i$ is either $\bigvee_{t_i>s_i}$ or $\bigwedge_{t_i<s_i}$, we may change an order of the $\Xi_i$'s without changing the value of the expression.
\end{lemma}

\begin{proof}
It holds in the case $n=2$ by Lemma~\ref{lem:help2}. If we fix any coordinate in $F$, we get an $(n-1)$-dimensional pseudo spectral resolution, so we can freely permute the $\Xi_2,\ldots,\Xi_n$ by the induction. Using this and some re-indexing, it is enough to prove for each $\bar{t}\in(\mathbb{R}\cup\{\pm\infty\})^n$ the equation
$$\bigvee_{s_1<t_1}\cdots \bigvee_{s_k<t_k}\bigwedge_{s_{k+1}<t_{k+1}}\cdots\bigwedge_{s_{n}<t_{n}} F(s_1,\ldots,s_n)=\bigwedge_{s_{k+1}<t_{k+1}}\cdots\bigwedge_{s_{n}<t_{n}} \bigvee_{s_1<t_1}\cdots \bigvee_{s_k<t_k}F(s_1,\ldots,s_n).$$

Let us rewrite it in an easier form: For each $\bar{s}\in(\mathbb{R}\cup\{\pm\infty\})^k$ and
$\bar{t}\in(\mathbb{R}\cup\{\pm\infty\})^{n-k}$,
$$
\bigvee_{\bar{q}\ll \bar{s}}\bigwedge_{\bar{r}\gg\bar{t}}F(\bar{q},\bar{r})= \bigwedge_{\bar{r}\gg\bar{t}}F(\bar{s},\bar{r}).
$$
(We have used the continuity property.)  For each $\bar{q}\ll \bar{s}$ and $\bar{r}\gg \bar{t}$, we shall prove the second inequality (the first one is trivial) in
\begin{equation}\label{eq:L2.2.1}
0\leq \bigwedge_{\bar{r}\gg\bar{t}}F(\bar{s},\bar{r})-\bigwedge_{\bar{r}\gg
\bar{t}}F(\bar{q},\bar{r})\leq F(\bar{s},\bar{r})-F(\bar{q},\bar{r}).
\end{equation}
If we prove~\eqref{eq:L2.2.1}, we are done, as the last expression clearly goes to $0$ as $\bar{q}$ goes to $\bar{s}$. We prove
\begin{equation}\label{eq:L2.2.2}
\bigwedge_{\bar{r}\gg\bar{t}}F(\bar{s},\bar{r})- \bigwedge_{\bar{r}\gg\bar{t}}F(\bar{q},\bar{r})= \bigwedge_{\bar{r}\gg\bar{t}}[F(\bar{s},\bar{r})-F(\bar{q},\bar{r})]
\end{equation}
which is clearly enough for us. We would like to apply the special case of Lemma~\ref{lem:help1}, but the infima are not taken over countable non-increasing sequences. The two infima on the left hand side may be in the obvious way (using monotony of $F$) rewritten to be taken over countable monotone sequences. In the case of right hand infimum, we have to verify that whenever $\bar{r}_0\leq \bar{r}_1$ then $F(\bar{s},\bar{r}_0)-F(\bar{q},\bar{r}_0)\leq F(\bar{s},\bar{r}_1)-F(\bar{q},\bar{r}_1)$ (which in fact proves the infimum is well defined). To achieve this, we cannot use the volume condition directly (as the occurring points of $\mathbb{R}^n$ are generally not vertices of some rectangle with edges parallel to axis), we have to find a sequence $\bar{s}=\bar{s}_0\geq\bar{s}_1\geq\cdots\geq\bar{s}_k=\bar{q}$ such that two consecutive elements differ at most in one coordinate.
Then we can write
$$
F(\bar{s},\bar{r}_0)-F(\bar{q},\bar{r}_0)=(F(\bar{s}_0, \bar{r}_0)-F(\bar{s_1},\bar{r}_0))+\cdots+(F(\bar{s}_{k-1}, \bar{r}_0)-F(\bar{s}_k,\bar{r}_0))
$$
and
$$
F(\bar{s},\bar{r}_1)-F(\bar{q},\bar{r}_1)=(F(\bar{s}_0, \bar{r}_1)-F(\bar{s_1},\bar{r}_1))+\cdots+(F(\bar{s}_{k-1}, \bar{r}_{1})-F(\bar{s}_k,\bar{r}_1)).
$$
Now $(F(\bar{s}_j,\bar{r}_0)-F(\bar{s}_{j+1},\bar{r}_0)) \leq(F(\bar{s}_j,\bar{r}_1)-F(\bar{s}_{j+1},\bar{r}_1))$, this could be proved finding another sequence $\bar{r}_0\leq\cdots\leq\bar{r}_1$ similar to $(\bar{s}_i)_i$ and using the volume condition in each step.

Hence, we can apply the particular case of Lemma~\ref{lem:help1}, to prove~\eqref{eq:L2.2.2} and so to finish the whole proof.
\end{proof}

Using Lemma \ref{le:+,-}, we can extend the notion of the difference operators $\Delta_i$ for $i=1,\ldots,n$ as follows. Let $\sigma_i,\tau_i\in \{+,-,\emptyset\}$ for $i=1,\ldots,n$. We set $-\infty=-\infty^+=-\infty^\emptyset$, $\infty=\infty^-=\infty^\emptyset$, moreover $-\infty^-$ and $\infty^+$ are not defined. For each $t\in \mathbb R\cup\{\pm\infty\}$, we put $t^\emptyset = t$. We linearly order these symbols as follows: If $s<t$, then $s^\sigma\le t^\tau$. If $s=t$, we assume $s^-\le s=s^\emptyset \le s^+$.

Lemma \ref{le:+,-} shows that we can define unambiguously expressions of the form $F(s^+,t^-,u,w^-,\infty)$, etc. Therefore, we can extend the domain of $F$ in the following sense. Let $\widehat{\mathbb R} :=\{r, r^+,r^-,r^\emptyset \colon r \in \mathbb R\}\cup \{-\infty, \infty\}$. Then we use the same symbol $F$ also for expressions $F(u_1,\ldots,u_n)$, where $u_1,\ldots,u_n \in \widehat{\mathbb R}$, that is, $F$ will be a function from $(\widehat{\mathbb R})^n$ into the MV-algebra $M$ or into an effect algebra $E$. Hence, let $i=1,\ldots,n$ be given. We extend formally $\Delta_i$ in the following way: Let $s,t\in \mathbb R\cup\{\pm \infty\}$, $\sigma_i,\tau_i\in \{-,+,\emptyset\}$, $s^{\sigma_i}\le t^{\tau_i}$. Then
$$
\Delta_i(s^{\sigma_i},t^{\tau_i})F(u_1,\ldots,u_n):= F(u_1,\ldots,u_{i-1},t^{\tau_i},u_{i+1},\ldots,u_n)-
F(u_1,\ldots,u_{i-1},s^{\sigma_i},u_{i+1},\ldots,u_n),
$$
where $u_1,\ldots,u_n \in \widehat{\mathbb R}$,
assuming that $s^{\sigma_i}$ and $t^{\tau_i}$ are defined.
Whence
\begin{equation}\label{eq:I}
\Delta_i(s^{\sigma_i},t^{\tau_i})\big(\Delta_j(s^{\sigma_j}, t^{\tau_j})F\big)= \Delta_j(s^{\sigma_j},t^{\tau_j})\big(\Delta_i(s^{\sigma_i},
t^{\tau_i})F\big)
\end{equation}
for each $i<j$.
More generally, let $(i_1,\ldots,i_n)$ be any permutation of $(1,\ldots,n)$,  $s_i\le t_i$ from $\mathbb R\cup\{\pm\infty\}$ be such that $s_i^{\sigma_i}\le t_i^{\tau_i}$ are defined.

Then
\begin{equation}\label{eq:II}
\Delta_1(s_1^{\sigma_1},t_1^{\tau_1})\big(\cdots \big( \Delta_n(s_n^{\sigma_n},t_n^{\tau_n})F\big)\cdots\big)=
\Delta_{i_1}(s_{i_1}^{\sigma_{i_1}},t_{i_1}^{\tau_{i_1}}) \big(\cdots\big( \Delta_{i_n}(s_{i_n}^{\sigma_{i_n}},t_{i_n}^{\tau_{i_n}})F\big) \cdots\big).
\end{equation}
Therefore, we can put unambiguously
$$
\Delta_1(s_1^{\sigma_1},t_1^{\tau_1})\cdots \Delta_n(s_n^{\sigma_n},t_n^{\tau_n})F:= \Delta_1(s_1^{\sigma_1},t_1^{\tau_1})\big(\cdots \big( \Delta_n(s_n^{\sigma_n},t_n^{\tau_n})F\big)\cdots\big).
$$

If $s\le t\le v$, $\sigma_i,\tau_i\in \{+,-,\emptyset\}$ such that $s^{\sigma_i}\le t^{\tau_i}\le v^{\tau_i}$, then
\begin{equation}\label{eq:III}
\Delta_i(s^{\sigma_i},v^{\tau_i})F= \Delta_i(s^{\sigma_i},t^{\tau_j})F + \Delta_i(t^{\tau_j},v^{\tau_i})F.
\end{equation}

We note that if $s\le t$, then $\Delta_i(s^-,t^{\tau_i}) =\Delta_i(s,t^{\tau_i})$ and $\Delta_i(s^{\sigma_i},t^-)=\Delta_i(s^{\sigma_i},t)$. Therefore, we have the following lemma.

\begin{lemma}\label{le:IV}
If $i_1<\cdots<i_k$ are elements of $\{1,\ldots,n\}$, $1\le k\le n$, then
\begin{equation}\label{eq:IV}
\Delta_{i_1}(s_{i_1}^{\sigma_{i_1}},t_{i_1}^{\tau_{i_1}})\cdots \Delta_{i_k}(s_{i_k}^{\sigma_{i_k}},t_{i_k}^{\tau_{i_k}})F(v_1,\ldots,v_n)\ge 0, \quad (v_1,\ldots,v_n)\in \mathbb R^n,
\end{equation}
where $\sigma_{i_l}$ and $\tau_{i_l}$ satisfy conditions to be the left hand side defined.
\end{lemma}

\begin{proof}
Due to (\ref{eq:I}), we can assume that $i_l=l$ for each $l=1,\ldots,k$. Choose sequences $(s^m_l)_m\searrow s_l$ and $(t^m_l)_m\searrow t_l$ for $l=1,\ldots,k$, $s^m_l\le t^m_l$ such that $s^m_l=s_l$ if $\sigma_l\in\{-,\emptyset\}$ and $t^m_l=t_l$ if $\tau_l\in \{-,\emptyset\}$, $m\ge 1$. Due to the volume condition, $\Delta_1(s^m_1,t^m_1)\cdots \Delta_k(s^m_k,t^m_k)(v_1,\ldots,v_n)\ge 0$. Expanding the left hand side and putting on the left hand side all members with the negative sign and on the right hand side all members with the positive sign. Let $\sum_{LS}(m)$ and $\sum_{RS}(m)$ be their sums. Then $\sum_{LS}(m)\le \sum_{RS}(m)$ for each $m\ge 1$ which yields $\bigwedge_m \sum_{LS}(m)\le \bigwedge_m\sum_{RS}(m)$. So we obtain on the left hand side the sum of all members of (\ref{eq:IV}) with the negative sign and on the right hand we get the sum of all members of (\ref{eq:IV}) with the positive sign which establishes (\ref{eq:IV}).
\end{proof}

\begin{lemma}\label{le:obser}
Let $F$ be an $n$-dimensional spectral resolution on a lexicographic effect algebra $E$, $n>1$, $i=1,\ldots,n$, and let $t_1\le t_2$ be from $\mathbb R\cup\{\pm \infty\}$. Then $F'(s_1,\ldots,s_{i-1},s_{i+1},\ldots,s_n):=
\Delta_i(t_1^{\sigma_i},t_2^{\tau_i})F(s_1,\ldots,s_n)$, $s_1,\ldots,s_{i-1},s_{i+1},\ldots,s_n \in \mathbb R$, where $\sigma_i$ and $\tau_i$ satisfy conditions to be $F'$ defined, is an $n-1$-dimensional pseudo spectral resolution.

Let $E=\Gamma_{ea}(\mathbb Z \lex G,(1,0))$ be perfect and $t_1\le t_2$, where $t_1,t_2\in \mathbb R$. Then
\begin{itemize}
\item[{\rm (i)}] $F'=\Delta_i(t_1,t_2)F$ has a characteristic point if and only if $t_1\le t^0_i$ and $t_2> t^0_i$;
\item[{\rm (ii)}] $F'=\Delta_i(t_1,t_2^+)F$ has a characteristic point if and only if $t_1\le t^0_i$ and $t_2\ge t^0_i$;
\item[{\rm (iii)}] $F'=\Delta_i(t^+_1,t^+_2)F$ has a characteristic point if and only if $t_1<t^0_i$ and $t_2\ge t^0_i$;
\item[{\rm (iv)}] $F'=\Delta_i(-\infty,t_2)F$ has a characteristic point if and only if $t_2>t^0_i$;
\item[{\rm (v)}] $F'=\Delta_i(-\infty,t_2^+)F$ has a characteristic point if and only if $t_2\ge t^0_i$;
\item[{\rm (vi)}] $F'=\Delta_i(t_1,\infty)F$ has a characteristic point if and only if $t_1< t^0_i$;
\item[{\rm (vii)}] $F'=\Delta_i(t_1^+,\infty)F$ has a characteristic point if and only if $t_1< t^0_i$;
\item[{\rm (viii)}] $F'=\Delta_i(-\infty,\infty)F$ has always a characteristic point.
\end{itemize}
In either case, if $(t^0_1,\ldots,t^0_n)$ is a unique characteristic point of $F$, then $(t^0_1,\ldots, t^0_{i-1},t^0_{i+1},\ldots, t^0_n)$ is a unique characteristic point of $F'$.
Moreover, $F'$ satisfies {\rm (v)} of Definition {\rm \ref{de:n-dim}}.

In particular, $\Delta_i(t^0_i,t^{0+}_i)F$ has a unique characteristic point, namely $(t^0_1,\ldots, t^0_{i-1},t^0_{i+1},\ldots, t^0_n)$.
\end{lemma}

\begin{proof}
Without loss of generality, we can assume that $i=1$.

(1) First we deal with $F':=\Delta_1(t_1,t_2)F$. The volume condition is trivially satisfied, and using Lemma \ref{lem:help1}, we see that $F'$ satisfies conditions (ii), (iii), and (iv').

(vi). Assume that $E$ is perfect and let $(t^0_1,\ldots,t^0_n)$ be the unique characteristic point of $F$. We note that $F'(v'_2,\ldots,v'_n)\in \Rad(E)'$ iff $F(t_2,v'_2,\ldots,v'_n)\in \Rad(E)'$ and $F(t_1,v'_2,\ldots,v'_n)\in \Rad(E)$. The first condition entails $(t_2,v'_2,\ldots,v'_n)\gg (t^0_1,\ldots,t^0_n)$, i.e. $t_2>t^0_1$ and $(v'_2,\ldots,v'_n)\gg (t^0_2,\ldots,t^0_n)$, and therefore, from the second condition we conclude $t_1\le t^0_1$. In other words, $F'$ has a characteristic point iff $t_1\le t^0_1$ and $t_2> t^0_1$. Put $T_1(F'):=\{(v_2,\ldots,v_n) \in \mathbb R^{n-1}\colon F'(v_2,\ldots,v_n)\in \Rad(E)\}\ne \emptyset$. Let $(v'_2,\ldots,v'_n)\in T_1(F')$. Then $F(t_2,v'_2,\ldots,v'_n)\in \Rad(E)'$ and $F(t_1,v'_2,\ldots,v'_n)\in \Rad(E)$. Whence $(t_2,v'_2,\ldots,v'_n)\gg (t^0_1,\ldots,t^0_n)$ and $(v'_2,\ldots,v'_n)\gg (t^0_2,\ldots,t^0_n)$. Then $F(t_2,v'_2,\ldots,v'_n)\ge \bigwedge\{F(t_1,\ldots,t_n)\colon F(t_1,\ldots,t_n)\in \Rad(E)'\}$, so that $a_2:=\bigwedge\{F(t_2,v'_2,\ldots,v'_n)\colon F'(v'_2,\ldots,v'_n)\in \Rad(E)'\}$ exists in $E$ and it belongs also to $\Rad(E)'$. Since $\{F(t_1,v'_2,\ldots,v'_n)\in \Rad(E)\}$ has a lower bound in $\Rad(E)$, the element $a_1:=\bigwedge\{F(t_1,v'_2,\ldots,v'_n)\in \Rad(E)\}$ exists in $E$ and it belongs to $\Rad(E)$. Applying Lemma \ref{lem:help1}, we see $a_2-a_1 =\bigwedge \{F'(v'_2,\ldots,v'_n)\colon F'(v'_2,\ldots,v'_n)\in \Rad(E)'\}$ and, consequently, (vi) holds for $F'$. In addition, $\bigwedge\{(v'_2,\ldots,v'_n)\colon F'(v'_2,\ldots,v'_n)\in \Rad(E)'\}= (t^0_2,\ldots,t^0_n)$, so that $(t^0_2,\ldots,t^0_n)$ is a unique characteristic point of $F'$ and it exists iff $t_1\le t^0_1$ and $t_2>t^0_1$.

(2) Let $F'=\Delta_1(t_1^\sigma,t_2^\tau)F$. Using Lemmas \ref{lem:help1}, \ref{le:+,-}, \ref{le:IV}, we can follow basic ideas of part (1).

(i) If $F'=\Delta_1(t_1,t_2)F$, the unique characteristic point was described in (1).

(ii) Let $F'=\Delta_1(t_1,t_2^+)F$. Then we have $(u'_2,\ldots, u'_n)\in T_1(F')$ iff $F(t^+_2, u'_2,\ldots, u'_n)\in \Rad(E)'$ and $F(t_1,u'_2,\ldots, u'_n)
\in \Rad(E)$. Take a sequence of real numbers $(t^m)_m\searrow t_2$ and $t^m> t_2$ for each $m\ge 1$. Then $F(t^+_2,u'_2,\ldots, u'_n)\le F(t^m,u'_2,\ldots, u'_n)\in \Rad(E)'$ for each $m\ge 1$. Therefore, $t^m>t^0_1$, $(u'_2,\ldots, u'_n)\gg (t^0_2,\ldots,t^0_n)$, and $t_2=\lim_m t^m\ge t^0_1$. In addition, $F(t_1,u'_2,\ldots, u'_n)\in \Rad(E)$ entails $t_1\le t^0_1$. This implies that $(t^0_2,\ldots,t^0_n)$ is a characteristic point of $F'$ and it exists iff $t_1\le t^0_1$ and $t_2\ge t^0_1$.

In addition,
\begin{align*}
a_2&:= \bigwedge_{(u'_2,\ldots, u'_n)}F(t^+_2,u'_2,\ldots, u'_n)
=\bigwedge_{(u'_2,\ldots, u'_n)}\bigwedge_m F(t^m_2,u'_2,\ldots, u'_n)\in \Rad(E)',\\
a_1&= \bigwedge_{(u'_2,\ldots, u'_n)}F(t_1,u'_2,\ldots, u'_n)\in \Rad(E),\\
a&=a_2-a_1 = \bigwedge_{(u'_2,\ldots, u'_n)}F'(u'_2,\ldots, u'_n)\in \Rad(E)'.
\end{align*}

(iii) Let $F'=\Delta_1(t^+_1,t^+_2)F$. Then we conclude $(u'_2,\ldots, u'_n)\in T_1(F')$ iff $F(t^+_2,u'_2,\ldots, u'_n)\in \Rad(E)'$ and $F(t^+_1,u'_2,\ldots, u'_n)\in \Rad(E)$. The first condition entails by (ii) $t_2\ge t^0_1$, $(u'_2,\ldots, u'_n)\gg (t^0_2,\ldots,t^0_n)$. Let $(t^m)_m\searrow t_1$. The second condition gets  there is $m_0$ such that $F(t^+_1,u'_2,\ldots, u'_n)\le F(t^m,u'_2,\ldots, u'_n)\in \Rad(E)$ for each $m\ge m_0$. Hence, $t_m\le t^0_1$ for $m\ge m_0$, so that $t_1=\lim_m t_m\le t^0_1$ which entails $(t^0_2,\ldots,t^0_n)$ is a characteristic point.
Moreover,
\begin{align*}
a_2&:=
\bigwedge_{(u'_2,\ldots, u'_n)}\bigwedge_m F(t^m_2,u'_2,\ldots, u'_n)\in \Rad(E)',\\
a_1&:= \bigwedge_{(u'_2,\ldots, u'_n)}F(t_1,u'_2,\ldots, u'_n)= \bigwedge_{(u'_2,\ldots, u'_n)}\bigwedge_m F(t^m,u'_2,\ldots, u'_n) \in \Rad(E),\\
a&:=a_2-a_1 = \bigwedge_{(u'_2,\ldots, u'_n)}F'(u'_2,\ldots, u'_n)\in \Rad(E)'.
\end{align*}
In a similar way we proceed in cases (iv)--(viii), so that condition (vi) holds also in (2).

The last statement on $\Delta_i(t^0_i,t^{0+}_i)F$ follows from (2)(ii).
\end{proof}

\begin{remark}\label{re:Delta}
The latter result can be extended: Let $i_1<\cdots<i_k$ be integers from $\{1,\ldots,n\}$. If $F$ is an $n$-dimensional (pseudo) spectral resolution, then the left hand side of {\rm (\ref{eq:IV})} defines an $(n-k)$-dimensional pseudo spectral resolution. In particular, $\Delta_{i_1}(t^0_{i_1},t^{0+}_{i_1})\cdots \Delta_{i_k}(t^0_{i_k},t^{0+}_{i_k})F$ has a unique characteristic point $(t^0_{j_1},\ldots,t^0_{j_{n-k}})$, where $j_1<\cdots<j_{n-k}$ is the natural ordering of the set $\{j_1,\ldots,j_{n-k}\} = \{1,\ldots,n\}\setminus\{i_1,\ldots,i_k\}$, whenever $1\le k<n$.
\end{remark}

Suppose we have an $n$-dimensional spectral resolution $F$ on a lexicographic effect algebra $E$ with only finitely many characteristic points. We will construct an $n$-dimensional observable $x$ extending $F$. We will use the abbreviation for points in $\mathbb R^n$ in the form of vectors, so that $\bar t=(t_1,\ldots,t_n)$, $\bar s=(s_1,\ldots,s_n)$, and $\bar r = (r_1,\ldots,r_n)$, etc. In analogy with the case of perfect MV-algebras, lets chop $\mathbb{R}^n$ by hyperplanes $\{\bar{t}\in\mathbb{R}^n\colon t_i=s\}$ whenever $s$ is the $i$-th coordinate of some characteristic point. So we obtain a cover of $\mathbb{R}^n$ by a collection of cells $B_j$, $j\in J$, where each $B_j$ is of the form $\prod_{i=1}^n(t^j_i,s^j_i]$, where some $s^j_i$ may equals $\infty$, in which case we, of course, replace the bracket $]$ by $)$, and similarly some $t^j_i$ can be also $-\infty$.

Now, for each $B_j$ we define an $n$-dimensional pseudo spectral resolution $F_j$ which has its support on $B_j$. To avoid to many indexes, let $B=(\bar{t},\bar{s}]$ be one of the blocks $B_j$'s (we will leave the index $j$): Let us define

\begin{equation}\label{eq:F,B}
F_B(r_1,\ldots,r_n)= \left\{\begin{array}{ll}
\Delta_1(t_1^+,\min\{s_1,r_1\})\cdots \Delta_n(t_n^+,\min\{s_n,r_n\})F & \mbox{if} \ \bar{t}\leq \bar{r},\\
0 & \mbox{if}\ \bar{t}\nleq \bar{r}.
\end{array}
\right.
\end{equation}

Lemma~\ref{lem:help3} guarantees that, for $\bar{t}\leq \bar{r}$, the value $F_B(\bar{r})$ equals infimum from $V_{\bar{t_i}}^{\bar{r}}(F)$, where $\bar{t}\ll\bar{t}_i$ and $(\bar{t}_i)_i\searrow \bar{t}$.
Then $F_B$ is monotone.

The volume condition: Take real numbers $a_i\le b_i$ for $i=1,\ldots,n$. If $b_i<t^+_i$ for some $i$, then $\Delta_i(a_i,b_i)F_B=0$, so that $\Delta_1(a_1,b_1)\cdots \Delta_n(a_n,b_n)F_B= 0$. So let $t^+_i\le b_i$ for each $i=1,\ldots,n$. Then it is possible to show $\Delta_1(a_1,b_1)\cdots \Delta_n(a_n,b_n)F_B=
\Delta'_1(a_1,b_1)\cdots \Delta'_n(a_n,b_n)F$, where
$$
\Delta'_i(a_i,b_i)= \left\{\begin{array}{ll}
\Delta_i(t^+_i,\min\{s_i,b_i\}) & \mbox{if} \ a_i< t^+_i,\\
\Delta_i(\min\{s_i,a_i\},\min\{s_i,b_i\}) & \mbox{if}\ t^+_i\le a_i,
\end{array}
\right.
$$
which entails $\Delta_1(a_1,b_1)\cdots \Delta_n(a_n,b_n)F_B\ge 0$.

\section{ $n$-dimensional Spectral Resolutions on a Perfect MV-algebra}

If $n=1$, then $F$ can be extended to a one-dimensional observable due to \cite[Thm 4.8]{DDL}, and if $n=2$, it was proved in Theorem \ref{th:TwoMV}. The same is true also if $F$ is a pseudo spectral resolution, then we can find an observable on $[0,u^0_n]$, $n=1,2$, where $u^0_n=F(\infty,\ldots,\infty)$. We will suppose that every $i$-dimensional pseudo spectral resolution, where $1\le i\le n-1$, can be extended to an $i$-dimensional observable on $[0,u^i_0]$. Our aim is to show that then this is true also for $i=n$.

Let $\mathbf t_0=(t^0_1,\ldots,t^0_n)$ be a unique characteristic point of an $n$-dimensional spectral resolution $F$. We divide $\mathbb R^n$ by hyper-planes going through $\mathbf t_0$ whose normals are parallel with the main axes. We obtain $2^n$ blocks.
For any $n$-tuple $(i_1,\ldots,i_n)\in \{0,1\}^n$, we define $B_{i_1,\ldots,i_n}$ as follows
$B_{i_1,\ldots,i_n}=A_{i_1}\times \cdots \times A_{i_n}$, where
\begin{equation}\label{eq:Aij}
A_{i_j}= \left\{\begin{array}{ll}
(-\infty,t^0_{i_j}]  & \mbox{if}\ i_j=0,\\
(t^{0}_{i_j},\infty) & \mbox{if} \ i_j=1.
\end{array}
\right.
\end{equation}
Then we define new difference operators $ \Delta_j^{i_j}$, $i_j\in\{0,1\}$, $j=1,\ldots,n$, as

\begin{equation}\label{eq:Delta}
\Delta_j^{i_j}= \left\{\begin{array}{ll}
\Delta_j(-\infty,t_j^0)  & \mbox{if}\ i_j=0,\\
\Delta_j(t^{0+}_j,\infty) & \mbox{if} \ i_j=1.
\end{array}
\right.
\end{equation}
In addition, we set
$$
u_{i_1,\ldots,i_n}= u^n_{i_1,\ldots,i_n}=\Delta^{i_1}_1\cdots \Delta^{i_n}_nF.
$$
Due to Lemma \ref{le:obser}, $u_{i_1,\ldots,i_n}\in \Rad(M)$.
Let $F_{i_1,\ldots,i_n}$ be a function defined by (\ref{eq:F,B}) for $B=B_{i_1,\ldots,i_n}$, then $F_{i_1,\ldots,i_n}$ is an $n$-dimensional pseudo spectral resolution on $[0,u_{i_1,\ldots,i_n}]\subseteq \Rad(M)$.

Let $n_1<\cdots< n_k$ be any system of $k$ integers from $\{1,\ldots,n\}$, where $0<k<n$, $\{m_1,\ldots,m_{n-k}\}=\{1,\ldots,n\}\setminus \{n_1,\ldots,n_k\}$, and $m_1<\cdots<m_{n-k}$.
Let us determine a $k$-dimensional subspace going through the characteristic point $\mathbf t_0$ and determined by vectors parallel with axes $x_{n_1},\ldots,x_{n_k}$. For every fixed $n_1<\cdots<n_k$, we put
$$
F_{n_1,\ldots,n_k}(t_{n_1},\ldots,t_{n_k})= \Delta_{m_1}(t^0_{m_1},t^{0+}_{m_1}) \cdots \Delta_{m_{n-k}}(t^0_{m_{n-k}},t^{0+}_{m_{n-k}})F(t_1,\ldots,t_n), \quad (t_{n_1},\ldots,t_{n_k})\in \mathbb R^k.
$$

According to Lemma \ref{le:obser} and Remark \ref{re:Delta}, $F_{n_1,\ldots,n_k}$ is a $k$-dimensional pseudo spectral resolution. If we set
$$
u^n_{n_1,\ldots,n_k} =\bigvee_{t_{n_1},\ldots,t_{n_k}} F_{n_1,\ldots,n_k}(t_{n_1},\ldots,t_{n_k}),
$$
then
$$
u^n_{n_1,\ldots,n_k}= \widehat \Delta^n_1(n_1,\ldots,n_k)\cdots\widehat \Delta^n_n(n_1,\ldots,n_k)F,
$$
where
$$
\widehat \Delta^n_i(n_1,\ldots,n_k)= \left\{\begin{array}{ll}
\Delta_{n_j}(-\infty,\infty)  & \mbox{if}\ i= n_j \mbox{ for some } j=1,\ldots,k,\\
\Delta_{m_j}(t^{0}_{m_j},t^{0+}_{m_j}) & \mbox{if } i =m_j \mbox{ for some } j=1,\ldots,n-k,
\end{array}
\quad i=1,\ldots,n.
\right.
$$
Due to Lemma \ref{le:obser}, $F_{n_1,\ldots,n_k}$ is a $k$-dimensional spectral resolution with a unique characteristic point $(t^0_{n_1},\ldots,t^0_{n_k})$ on the interval MV-algebra $[0,u_{n_1,\ldots,n_k}]=\Gamma(\mathbb Z \lex G, (1,-g_0))$, where $(1,-g_0)=u_{n_1,\ldots,n_k}$, with $u_{n_1,\ldots,n_k}\in \Rad(M)'$ which is ``almost close" to be a perfect MV-algebra for all $n_1<\cdots<n_k$, where $1\le k \le n$.

The mapping $F_{i_1,\ldots,i_n}$, where $i_1,\ldots,i_n \in \{0,1\}$, is in fact an $n$-dimensional spectral resolution on the $\sigma$-complete interval MV-algebra $[0,u_{i_1,\ldots,i_n}]$ with $u_{i_1,\ldots,i_n}\in \Rad(M)$. Due to \cite[Thm 5.1]{DvLa3}, there is an $n$-dimensional observable $x_{i_1,\ldots,i_n}$ on $[0,u_{i_1,\ldots,i_n}]$ which is an extension of $F_{i_1,\ldots,i_n}$. Moreover, $x_{i_1,\ldots,i_n}(A)= x_{i_1,\ldots,i_n}(A\cap B_{i_1,\ldots,i_n})$, $A \in \mathcal B(\mathbb R^n)$, and $x_{i_1,\ldots,i_n}(\mathbb R^n)=u_{i_1,\ldots,i_n}$.
On the other hand, let $x_{n_1,\ldots,n_k}$ be a $k$-dimensional observable uniquely determined by the $k$-dimensional pseudo spectral resolution $F_{n_1,\ldots,n_k}$, $n_1<\cdots<n_k$, $k=1,\ldots,n-1$, on the interval algebra $[0,u_{n_1},\ldots,{n_k}]$, see the induction assumption from the beginning of the section. In addition, $x_{n_1,\ldots,n_k}(\mathbb R^k)=u^n_{n_1,\ldots,n_k}$.

It is possible to show that
\begin{equation}\label{eq:u+n}
x_{n_1,\ldots,n_k}(\{t^0_{n_1},\ldots,t^0_{n_k}\})=
\Delta_1(t_1^0,t_1^{0+})\cdots \Delta_n(t_n^0,t_n^{0+})F=:u^n_\emptyset
\end{equation}
for all $n_1,\ldots,n_k \in \{1,\ldots,n\}$. We notify $u^n_\emptyset \in \Rad(M)'$. We assert

\begin{equation}\label{eq:f3.1}
1= \sum_{i_1,\ldots,i_n\in \{0,1\}} u^n_{i_1,\ldots,i_n} +\sum_{k=1}^{n-1}(-1)^{n-k+1}\big(\sum_{n_1<\cdots<n_k\le n} u^n_{n_1,\ldots,n_k}\big) + (-1)^{n+1}u^n_\emptyset.
\end{equation}

We prove (\ref{eq:f3.1}). If $n=1$, the formula follows from \cite[Thm 4.8]{DDL}, and if $n=2$, it was proved in Theorem \ref{th:TwoMV}. We denote $\Delta_i(t_i^0):=\Delta_i(t_i^0,t_i^{0+})$, $i=1,\ldots,n$.
If we fix the last coordinate, we have the $n$-dimensional pseudo observable $F_t(t_1,\ldots,t_n)=F(t_1,\ldots,t_n,t)$, $(t_1,\ldots,t_n)\in \mathbb R^n$, so for it we have formula (\ref{eq:f3.1}) with $u_{i_1,\ldots,i_n}^{n,t}$ and $u^{n,t}_{n_1,\ldots,n_k}$. If we put $t=\infty$, then $F_\infty$ is also an $n$-dimensional spectral resolution.
Using the induction holding for $F_\infty$, we have
\begin{align*}
1&= \Delta_1(-\infty,\infty)\cdots \Delta_n(-\infty,\infty)\Delta_{n+1}(-\infty,\infty)F= \Delta_1(-\infty,\infty)\cdots \Delta_n(-\infty,\infty)F_\infty\\
&=\sum_{i_1,\ldots,i_n\in \{0,1\}} u^{n,\infty}_{i_1,\ldots,i_n} +\sum_{k=1}^{n-1}(-1)^{n-k+1}\big(\sum_{n_1<\cdots<n_k\le n} u^{n,\infty}_{n_1,\ldots,n_k}\big) + (-1)^{n+1}u^{n,\infty}_\emptyset\\
&=\big\{\sum_{i_1,\ldots,i_n}\Delta_1^{i_1}\cdots\Delta_n^{i_n} + \sum_{k=1}^{n-1}(-1)^{n-k+1}\big(\sum_{n_1<\cdots<n_k\le n}\widehat \Delta_{1}(n_1,\ldots,n_k)\cdots\widehat\Delta_{n}(n_1,\ldots,n_k)\big)\\
&\quad + (-1)^{n+1}\Delta_1(t_1^0)\cdots \Delta_n(t_n^0)\big\}F_\infty\\
&=\big\{\sum_{i_1,\ldots,i_n}\Delta_1^{i_1}\cdots\Delta_n^{i_n} + \sum_{k=1}^{n-1}(-1)^{n-k+1}\big(\sum_{n_1<\cdots<n_k\le n}\widehat \Delta_{1}(n_1,\ldots,n_k)\cdots\widehat\Delta_{n}(n_1,\ldots,n_k)\big)\\
&\quad + (-1)^{n+1}\Delta_1(t_1^0)\cdots \Delta_n(t_n^0)\big\}\Delta_{n+1}(-\infty,\infty)F\\
&= \big\{\sum_{i_1,\ldots,i_n}\Delta_1^{i_1}\cdots\Delta_n^{i_n}+ \sum_{k=1}^{n-1}(-1)^{n-k+1}\big(\sum_{n_1<\cdots<n_k\le n}\widehat \Delta_{1}(n_1,\ldots,n_k)\cdots\widehat\Delta_{n}(n_1,\ldots,n_k)\big)\\
&\quad + (-1)^{n+1}\Delta_1(t_1^0)\cdots \Delta_n(t_n^0)\big)
\big\}\big(\Delta_{n+1}(-\infty,t_{n+1}^0)+\Delta_{n+1}(t_{n+1}^0,t_{n+1}^{0+}) + \Delta_{n+1}(t_{n+1}^{0+},\infty)\big)F\\
&= \big\{\sum_{i_1,\ldots,i_n}\Delta_1^{i_1}\cdots\Delta_n^{i_n} + \sum_{k=1}^{n-1}(-1)^{n-k+1}\big(\sum_{n_1<\cdots<n_k\le n}\widehat \Delta_{1}(n_1,\ldots,n_k)\cdots\widehat\Delta_{n}(n_1,\ldots,n_k)\big)\\
&\quad + (-1)^{n+1}\Delta_1(t_1^0)\cdots \Delta_n(t_n^0)\big)\big\}
(\Delta_{n+1}^0 + \Delta_{n+1}(t_{n+1}^0) +\Delta_{n+1}^1)F\\
&=  \big\{\sum_{i_1,\ldots,i_n,i_{n+1}}\Delta_1^{i_1}\cdots\Delta_n^{i_n} \Delta_{n+1}^{i_{n+1}} + \sum_{i_1,\ldots,i_n}\Delta_1^{i_1}\cdots\Delta_n^{i_n} \Delta_{n+1}(t_{n+1}^0)\\
&\quad + \sum_{k=1}^{n-1}(-1)^{n-k+1}\big(\sum_{n_1<\cdots<n_k\le n}\widehat \Delta_{1}(n_1,\ldots,n_k)\cdots\widehat\Delta_{n}(n_1,\ldots,n_k) \Delta_{n+1}(-\infty,\infty)\big)\\ &\quad +(-1)^{n+1}\Delta_1(t_1^0)\cdots \Delta_n(t_n^0) \Delta_{n+1}(-\infty,\infty) \big\}F\\
&= (\text{A})+(\text{B})+(\text{C}) + (\text{D}),
\end{align*}
where
\begin{eqnarray*}
(\text{A})&=\sum_{i_1,\ldots,i_n,i_{n+1}}\Delta_1^{i_1}\cdots\Delta_n^{i_n} \Delta_{n+1}^{i_{n+1}}F,
\end{eqnarray*}
\begin{eqnarray*}
(\text{B})&=\sum_{i_1,\ldots,i_n}\Delta_1^{i_1}\cdots\Delta_n^{i_n} \Delta_{n+1}(t_{n+1}^0)F\\
&= \big(\Delta_1(-\infty,\infty)-\Delta_1(t_1^0)\big)\cdots \big(\Delta_n(-\infty,\infty)-\Delta_n(t_n^0)\big)
\Delta_{n+1}(t_{n+1}^0)F\\
&=  \sum_{k=1}^{n}(-1)^{n-k} \sum_{n_1<\cdots<n_k\le n}  u^{n+1}_{n_1,\ldots,n_k} + (-1)^{n+2} u^{n+1}_\emptyset,
\end{eqnarray*}
\begin{eqnarray*}
(\text{C})&= \sum_{k=1}^{n-1}(-1)^{n-k+1}\big(\sum_{n_1<\cdots<n_k\le n}\widehat \Delta_{1}(n_1,\ldots,n_k)\cdots\widehat\Delta_{n}(n_1,\ldots,n_k) \Delta_{n+1}(-\infty,\infty)F\big)\\ &\quad
= \sum_{k=1}^{n-1}(-1)^{n-k+1}\sum_{n_1<\cdots<n_k\le n} u^{n+1}_{n_1,\ldots,n_k,n+1},
\end{eqnarray*}
and
\begin{eqnarray*}
(\text{D})&= (-1)^{n+1}\Delta_1(t_1^0)\cdots \Delta_n(t_n^0) \Delta_{n+1}(-\infty,\infty)F = (-1)^{n+1}u^{n+1}_{n+1}.
\end{eqnarray*}
Then
\begin{align*}
1=(\text{A})&+(\text{B})+(\text{C}) + (\text{D})\\
&= \sum_{i_1,\ldots,i_{n+1}}u^{n+1}_{i_1,\ldots,i_{n+1}} +
\sum_{k=1}^n (-1)^{n-k+2}\big(\sum_{i_1,\ldots,i_k\le n+1}u^{n+1}_{n_1,\ldots,n_k} \big) +(-1)^{n+2}u^{n+1}_\emptyset.
\end{align*}

We define a mapping $x:\mathcal B(\mathbb R^n) \to M$ corresponding to $F$, as follows

\begin{equation}\label{eq:formula}
\begin{split}
x(A)=&\sum_{i_1,\ldots,i_n\in \{0,1\}} x_{i_1,\ldots,i_n}(A) +\sum_{k=1}^{n-1}(-1)^{n-k+1}\big(\sum_{n_1<\cdots<n_k\le n}x_{n_1,\ldots,n_k}
(\pi^n_{n_1,\ldots,n_k}(A))\big)\\
 &+ (-1)^{n+1}u^n_\emptyset \chi_{(t^0_1,\ldots,t^0_n)}(A), \quad A \in \mathcal B(\mathbb R^n),
\end{split}
\end{equation}
where $\pi^n_{n_1,\ldots,n_k}$ is the projection from $\mathbb R^n$ onto $\mathbb R^k$ given by $\pi^n_{n_1,\ldots,n_k}(t_1,\ldots,t_n)= (t_{n_1},\ldots,t_{n_k})$ for each $(t_1,\ldots,t_n)\in \mathbb R^n$.  We note that the latter expression generalizes (\ref{eq:obser2}) for $n=2$ and $n=1$ from \cite[Thm 4.8]{DDL}.

We assert that $x$ can be expressed as the sum of positive elements of $M$ as follows.

\begin{lemma}\label{le:positive}
The function $x$ defined by {\rm (\ref{eq:formula})} takes positive values on each $A \in \mathcal B(\mathbb R^n)$.
\end{lemma}

\begin{proof}
Define $H_i$, $i=1,\ldots,n$, to be the hyperplane passing through the characteristic point and orthogonal to the $x_i$-th axis.

\begin{equation}\label{eq:formula1}
\begin{split}
x(A)=\sum_{i_1,\ldots,i_n\in \{0,1\}}x_{i_1,\ldots,i_n}(A)&+x_{2,\ldots,n}\big((\pi^n_{2,\ldots,n}(A\cap H_1))\big)
+x_{1,3,\ldots,n}\big(\pi^n_{1,3,\ldots,n}(A\cap (H_2\setminus H_1))\big)\\
&+\cdots +x_{1,2,\ldots,n-1}\big(\pi^n_{1,\ldots,n-1}(A\cap (H_n\setminus(\bigcup_{j<n}H_j)))\big),\quad A \in \mathcal B(\mathbb R^n).
\end{split}
\end{equation}
Let $A\in \mathcal B(\mathbb R^n)$ be fixed. Denote by $x^i(B):=x_{1,\ldots,i-1,i+1,\ldots,n}(\pi^n_{1,\ldots,i-1,i+1,\ldots,n}(A\cap B))$, $B \in \mathcal B(\mathbb R^n)$. Then $x^i$ is additive and subtractive on $M$ i.e. $B_1\subseteq B_2$ yields $x^i(B_2\setminus B_1)= x^i(B_2)- x^i(B_1)$. Hence,
we have
$$
x^i(H_i \setminus \bigcup_{j<i}H_j)= x^i(H_i\cap H_1^c\cap \cdots\cap H_{i-1}^c)= x^i(H_i)+ \sum_{k=1}^{i-1} (-1)^k\sum_{n_1<\cdots<n_k<i} x^i(H_i \cap\bigcap_{l\le k} H_{n_l}),
$$
where $^c$ denotes the set complement. Let $\{m_1,\ldots,m_{n-k-1}\} = \{1,\ldots,n\}\setminus \{i,n_1,\ldots,n_k\}$ and $m_1<\cdots<m_{n-k-i}$.
Then
$$
x^i(H_i)=x_{1,\ldots,i-1,i+1,\ldots,n}(\pi^n_{1,\ldots,i-1,i+1,\ldots,n}(A))
$$
and
$$
x^i(H_i \cap\bigcap_{l\le k} H_{n_l})= \left\{\begin{array}{ll}
x_{m_1,\ldots,m_{n-k-1}}(\pi^n_{m_1,\ldots,m_{n-k-1}}(A)) & \mbox{if}\ k<n-1,\\
\chi_{t_1^0,\ldots,t_n^0}(A)u^n_\emptyset & \mbox{if}\ k=n-1.
\end{array}
\right.
$$
Therefore, both right hand sides of formulas (\ref{eq:formula}) and (\ref{eq:formula1}) for $x(A)$ are identical and $x(A)\ge 0$.
\end{proof}

Now, we see that the mapping $x$ takes values in $M$, it is additive, monotone, and $x(\mathbb R^n)= 1$ due to (\ref{eq:f3.1}). In addition, similarly as for $n=2$, we can show that $\big(x(A_i)\big)_i\nearrow x(A)$ whenever $(A_i)_i\nearrow A$. In other words, $x$ is an $n$-dimensional observable on $M$.

In what follows, we show that $x$ is an extension of $F$, that is,
$x((-\infty,t_1)\times\cdots \times (-\infty,t_n))=F(t_1,\ldots,t_n)$, $t_1,\ldots,t_n\in \mathbb R$.

\begin{lemma}\label{le:F(x)}
If $x$ is determined by {\rm (\ref{eq:formula})}, then $x$ is an extension of $F$.
\end{lemma}

\begin{proof}
Assume that $x$ is defined by (\ref{eq:formula}). We show that $x$ is an extension of $F$. Due to \cite[Thm 4.8]{DDL} and Theorem \ref{th:TwoMV}, the statement is true for $n=1$ and $n=2$. So let $n>2$.

If $(k_1,\ldots,k_n)$ is any permutation of $(1,\ldots,n)$, then $F'(t_1,\ldots,t_n):=F(t_{k_1},\ldots,t_{k_n})$, $t_1,\ldots,t_n \in \mathbb R$, is an $n$-dimensional spectral resolution whose characteristic point is $(t^0_{k_1},\ldots,t^0_{k_n})$.

It is evident that if $(t_1,\ldots,t_n)\in B_{0,\ldots,0}$, then $x((-\infty,t_1)\times \cdots\times (-\infty,t_n))=F(t_1,\ldots,t_n)$.

Now, let $i_1,\ldots,i_n\in \{0,1\}$. Assume that some $i_j=0$ and $i_k=1$. Let $(j_1,\ldots,j_n)$ be a permutation of $\{1,\ldots,n\}$ such that $i_{j_1}=\cdots= i_{j_k}=1$ and $i_{j_{k+1}}=\cdots=i_{i_{j_n}}=0$.
If $(t_1,\ldots,t_n)\in B_{i_1,\ldots,i_n}$, then due to (\ref{eq:Aij}), we have $(t_{j_1},\ldots,t_{j_n})\in B_{i_{j_1},\ldots,i_{j_n}}$. Therefore, if $(t_{j_1},\ldots,t_{j_n})\in B_{i_{j_1},\ldots,i_{j_n}}$ implies $F'(t_{j_1},\ldots,t_{j_n})=x((-\infty,t_{j_1})\times\cdots\times (-\infty,t_{j_n}))$, then $F(t_1,\ldots,t_n)=x((-\infty,t_1)\times \cdots\times (-\infty,t_n))$.

Let $B=(-\infty,t_1)\times \cdots\times (-\infty,t_{n})$. Assume $(t_1,\ldots,t_n)\in B_{1,0,\ldots,0}$. Then
\begin{align*}
x(B)&= F_{0,\ldots,0}(t^0_1,t_2,\ldots,t_n)+F_{1,0,\ldots,0}(t_1,\ldots,t_n)+ F_{2,\ldots,n}(t_2,\ldots,t_n)\\
&= \Delta_1(-\infty,t^0_1)\Delta_2(-\infty,t_2)\cdots \Delta_n(-\infty,t_n)F + \Delta_1(t_1^{0+},t_1)\Delta_2(-\infty,t_2)\cdots\Delta_n(-\infty,t_n)F\\
&\quad + \Delta_1(t^0_1,t^{0+}_1)\Delta_2(-\infty,t_2)\cdots \Delta_n(-\infty,t_n)F\\
&= \Delta_1(-\infty,t_1)\cdots\Delta_n(-\infty,t_n)F=F(t_1,\ldots,t_n).
\end{align*}

Due to the previous paragraph, if $(t_1,\ldots,t_n)\in B_{i_1,\ldots,i_n}$, where $i_k=1$ and $i_j=0$ for $j\ne k$, then $x(B)= F(t_1,\ldots,t_n)$; this holds for each $k=1,\ldots,n$.

Take $(t_1,\ldots,t_n) \in B_{1,1,0,\ldots,0}$ and denote $\Delta^3 = \Delta_3(-\infty,t_3)\cdots\Delta_n(-\infty,t_n)$. Then
\begin{align*}
x(B)&= F_{0,0,0,\ldots,0}(t^0_1,t^0_2,t_3,\ldots,t_n)
\quad\quad  (=)\ \Delta_1(-\infty,t^0_1)\Delta_2(-\infty, t^0_2)\Delta^3F\\
 &+ F_{0,1,0,\ldots,0}(t^0_1,t_2,t_3,\ldots,t_n) \quad\quad  (=)\ \Delta_1(-\infty,t^0_1)\Delta_2(t^{0+}_2,t_2)\Delta^3F\\
&+ F_{1,0,0,\ldots,0}(t_1,t^0_2,t_3,\ldots,t_n) \quad\quad (=)\
\Delta_1(t^{0+}_1,t_1)\Delta_2(-\infty, t^0_2)\Delta^3 F\\
&+F_{1,1,0,\ldots,0}(t_1,t_2,t_3,\ldots,t_n) \quad\quad (=)\
 \Delta_1(t^{0+}_1,t_1)\Delta_2(t^{0+}_2,t_2)\Delta^3F\\
&+ F_{1,3,4,\ldots,n}(t_1,t_3,t_4,\ldots,t_n) \quad\quad (=)\ \Delta_1(-\infty,t_1)\Delta_2(t^0_2,t^{0+}_2)\Delta^3F\\
&+ F_{2,3,4,\ldots,n}(t_2,t_3,t_4,\ldots,t_n) \quad\quad (=)\
 \Delta_1(t^0_1,t^{0+}_1)\Delta_2(-\infty,t_2)\Delta^3F\\
&-F_{3,4,\ldots,n}(t_3,t_4,\ldots,t_n) \hspace{6mm}\quad\quad (=)\
- \Delta_1(t^0_1,t^{0+}_1)\Delta_2(t^0_2,t^{0+}_2)\Delta^3 F\\ &=\Delta_1(-\infty,t_1)\Delta_2(-\infty,t_2)\Delta^3F
= F(t_1,\ldots,t_n).
\end{align*}

Finally, let $A=(-\infty,t_1)\times\cdots\times(-\infty,t_{n})$.
Due to the last paragraphs, we can assume that if $(t_1,\ldots,t_{n})\in B_{i_1,\ldots,i_k,0,\ldots,0}$, then $x(A)=F(t_1,\ldots,t_{n})$, where $1\le k< n$. Without loss of generality and for simplicity, we assume that $k=n-1$.
Whence, if $(t_1,\ldots,t_{n})\in B_{1,\ldots,1,0}$, then $x(A)= F(t_1,\ldots,t_{n})$.

Now, take $(t_1,\ldots,t_{n})\in B_{1,\ldots,1}$.
Thus, let $t_{i}>t^0_{i}$ for each $i=1,\ldots,n$. Express $A=(-\infty,t_1)\times \cdots\times (-\infty,t_{n})$ in the form $A=A_0\cup A_1$, where $A_0=(-\infty,t_1)\times \cdots\times (-\infty,t_{n-1})\times(-\infty,t^0_{n}]$ and $A_1=(-\infty,t_1)\times \cdots\times (-\infty,t_{n-1})\times(t^0_{n},t_{n})$.
We have $x(A_0)=x((-\infty,t_1)\times \cdots\times (-\infty,t_{n-1})\times(-\infty,t^0_{n}))+ x((-\infty,t_1)\times \cdots\times (-\infty,t_{n-1})\times \{t^0_{n}\})$. Due to the assumption, we have
$$
x((-\infty,t_1)\times \cdots\times (-\infty,t_{n-1})\times(-\infty,t^0_{n})) =F(t_1,\ldots,t_{n-1},t^0_{n})
$$
and
\begin{align*}
x((-\infty,t_1)\times \cdots\times (-\infty,t_{n-1})\times \{t^0_{n}\})&=F_{1,\ldots,{n-1}}(t_1,\ldots,t_{n-1})\\
&=\Delta_1(-\infty,t_1)\cdots \Delta_{n-1}(-\infty,t_{n-1})\Delta_{n}(t^0_{n},t^{0+}_{n})F,
\end{align*}
so that
$$
x(A_0)=F(t_1,\ldots,t_{n-1},t^0_{n})+F_{1,\ldots,{n-1}}
(t_1,\ldots,t_{n-1})= F(t_1,\ldots,t_{n-1},t^{0+}_{n}).
$$

To calculate $x(A_1)$, we use (\ref{eq:Delta}) and (\ref{eq:formula}): Let $1\le n_1<\cdots<n_k< n$ and $m_1<\cdots<m_{n-k}$, where $\{m_1,\ldots,m_{n-1-k}\}=\{1,\ldots,n-1\}\setminus \{n_1,\ldots,n_k\}$, we get

\begin{align*}
x(A_1)=& \sum_{i_1,\ldots,i_{n-1}\in\{0,1\}}x_{i_1,\ldots,i_{n-1},0}((-\infty, t_1)\times \cdots \times (-\infty,t_{n-1})\times (t^0_{n},t_{n}))\\
 &\quad +
\sum_{k=0}^{n-2}(-1)^{n-k}\big(\sum_{n_1<\cdots<n_k<n} x_{n_1,\ldots,n_k,n}((-\infty,t_{n_1})\times \cdots \times (-\infty, t_{n_k})\times (t^0_{n},t_{n})\big)\\
=& \sum_{i_1,\ldots,i_{n-1}\in\{0,1\}}\Delta^{i_1}_1\cdots \Delta^{i_{n-1}}_{n-1}\Delta_{n}(t^{0+}_{n},t_{n})F\\
&\quad +
\sum_{k=0}^{n-2}(-1)^{n-k}\big(\sum_{n_1<\cdots<n_k<n}
\Delta_{m_1}(t^0_{m_1},t^{0+}_{m_1}) \cdots \Delta_{m_{n-k}}(t^0_{m_{n-1-k}},t^{0+}_{m_{n-1-k}}) \Delta_{n}(t^{0+}_{n},t_{n})F\big)\\
=& \Big(\sum_{i_1,\ldots,i_{n-1}\in\{0,1\}}\Delta^{i_1}_1\cdots
\Delta^{i_{n-1}}_{n-1}
 +\sum_{k=1}^{n-2}(-1)^{n-k}\\
& \quad \big(\sum_{n_1<\cdots<n_k<n} \Delta_{m_1}(t^0_{m_1},t^{0+}_{m_1}) \cdots \Delta_{m_{n-k}}(t^0_{m_{n-1-k}},t^{0+}_{m_{n-1-k}})\big)\Big) \Delta_{n}(t^{0+}_{n},t_{n})F\\
=&\Delta_1(-\infty,t_1)\cdots\Delta_{n-1}(-\infty,t_{n-1}) \Delta_{n}(t^{0+}_{n},t_{n})F= F(t_1,\ldots,t_{n})-F(t_1,\ldots,t_{n-1},t^{0+}_{n}).
\end{align*}
Whence, $x(A)= x(A_0)+x(A_1)=F(t_1,\ldots,t_{n-1},t_{n})$.
\end{proof}

In the following result, we define an $n$-dimensional observable on $M$ which is an extension of $F$ in another way as in formula (\ref{eq:formula}).

On $\Delta_i$'s operators we define addition $+$ usually as $\big(\Delta_i(t,t')+\Delta_i(s,s')\big)F=\Delta_i(t,t')F+\Delta_i(s,s')F$.

\begin{lemma}\label{eq:Fg} Let $F$ be an $n$-dimensional spectral resolution on $E=\Gamma_{ea}(\mathbb Z \lex G,(1,0))$, where $G$ is a directed Dedekind monotone $\sigma$-complete po-group with interpolation.
For each $g\in\{-1,0,1\}^{\{1,\ldots,n\}}$, we define
\begin{align*}F_g(t_1,\ldots,t_n)=&\prod_{i\in g^{-1}(-1)}\Delta_i(-\infty,\min\{t^0_i,t_i\})\prod_{i\in g^{-1}(0)}\Delta_i(\min\{{t^0_i},t_i\},\min\{{t^0_i}^+,t_i\})\\
&\prod_{i\in g^{-1}(1)}\Delta_i(\min\{{t^0_i}^+,t_i\},t_i)F,
\end{align*}
where $(t^0_1,\ldots,t^0_n)$ is the characteristic point of $F$.
Then the mapping $F_g$ is an $n$-dimensional pseudo spectral resolution. Unless the case $g$ equals constantly zero, $F_g$ has its range included in $\mathrm{Rad}(E)$.
\end{lemma}

\begin{proof}

First of all, we establish the following useful Claim.

\vspace{3mm}
\noindent
{\it Claim.}  {\it Let $F':\mathbb R^n\to E$ be an $n$-dimensional pseudo spectral resolution on $E$. For every $i=1,\ldots,n$, each of the mappings
\begin{align*}
F'_{-1}(t_1,\ldots,t_n)&:=\Delta_i(-\infty,\min\{t^0_i,t_i\}) F'(t_1,\ldots,t_n),\\ F'_0(t_1,\ldots,t_n)&:=\Delta_i(\min\{{t^0_i},t_i\}, \min\{{t^0_i}^+,t_i\})F'(t_1,\ldots,t_n), \quad (t_1,\ldots,t_n)\in \mathbb R^n,\\
F'_1(t_1,\ldots,t_n)&:=\Delta_i(\min\{{t_i^0}^+,t_i\},t_i) F'(t_1,\ldots,t_n),
\end{align*}
is an $n$-dimensional pseudo spectral resolution. Moreover, if $(t^0_1,\ldots,t^0_n)$ is the characteristic point of $F'$, then $F'_{-1}$ and $F'_1$ have no characteristic but $F'_0$ has a unique characteristic point, namely $(t^0_1,\ldots,t^0_n)$.
}
\vspace{2mm}

Indeed, let $s_i\le s'_i$ for each $i=1,\ldots,n$.

(1) Then
\begin{equation}\label{eq:L2}
 \Delta_i(s_i,s_i')F'_{-1}=  \Delta_i(s_i,s'_i)\Delta_i(-\infty,\min\{t_i,t_i^0\})F'= \Delta_i(\min\{s_i,t_i^0\},\min\{s'_i,t_i^0\})F'.
\end{equation}

Therefore,
\begin{align*}\Delta_1(s_1,s'_1)\cdots \Delta_n(s_n,s'_n)F'_{-1}&= \Delta_1(s_1,s'_1)\cdots \Delta_{i-1}(s_{i-1},s'_{i-1})\Delta_i(\min\{s_i,t_i^0\},\min\{s'_i,t_i^0\})\\ &\quad \Delta_{i+1}(s_{i+1},s_{i+1})\cdots \Delta_n(s_n,s'_n)F' \ge 0.
\end{align*}

(2)
\begin{equation}\label{eq:L3}
 \Delta_i(s_i,s_i')F'_1=   \Delta_i(s_i,s'_i)\Delta_i(\min\{t_i,{t_i^0}^+\},t_i)F'= \Delta_i(\max\{s_i,{t_i^0}^+\},\max\{s'_i,{t_i^0}^+\})F',
\end{equation}
which yields as in (1)
$\Delta_1(s_1,s'_1)\cdots \Delta_n(s_n,s'_n)F'_1\ge 0$.

(3) There are two cases: (i) $t^0_i\in [s_i,s'_i)$ which implies
$$
\Delta_i(s_i,s'_i)F'_0=\Delta_i(s_i,s'_i)\Delta_i(\min\{{t^0_i},t_i\}, \min\{{t^0_i}^+,t_i\})F'= \Delta_i(t^0_i,{t^0_i}^+)F'.
$$
(ii)
$t^0_i\notin [s_i,s'_i)$ which implies
$$
\Delta_i(s_i,s_i')F'_0=\Delta_i(s_i,s'_i)\Delta_i(\min\{{t^0_i},t_i\}, \min\{{t^0_i}^+,t_i\})F'=0.
$$
In either case, we have $\Delta_1(s_1,s'_1)\cdots \Delta_n(s_n,s'_n)F'_0\ge 0$.

Altogether (1)--(3) show that every $F'_j$ for $j=-1,0,1$ satisfies the volume condition.

Applying Lemma \ref{lem:help1}, we see that every $F'_j$ satisfies conditions (ii), (iii), (iv)' of Definition \ref{de:n-dim}, so that $F'_j$ is an $n$-dimensional pseudo spectral resolution. In addition, from the form of $F'_{-1}$ and $F'_1$ we see that they have no characteristic point, whereas, $F'_0$ does have it, namely $(t^0_1,\ldots,t^0_n)$,
which finishes the proof of Claim.

Now, let $g:\{1,\ldots,n\}\to \{-1,0,1\}$ be non-zero. Claim shows that
$F_g$ can be obtained from an appropriate sequence of $n$ operators $\Delta$'s applied to a sequence of $n$ appropriate functions of type $F'_{-1},F'_0,$ and $F'_1$, consequently, $F_g$ is an $n$-dimensional pseudo spectral resolution from $\mathbb R^n$ into $\Rad(E)$ whenever $g\ne 0$.

If $g=0$, then $F_g(t_1,\ldots,t_n)=\Delta_1(\min\{{t^0_1},t_1\},\min\{{t^0_1}^+,t_1\})\cdots \Delta_n(\min\{{t^0_n},t_n\},\min\{{t^0_n}^+,t_n\})F$ is a two-valued mapping from $\mathbb R^n$ into $\{0,u^n_\emptyset\}$, where $u^n_\emptyset = \Delta_1(t_1^0,t_1^{0+})\cdots \Delta_n(t_n^0,t_n^{0+})F\in \Rad(E)'$, see (\ref{eq:u+n}). Hence, $F_g$ is an $n$-dimensional pseudo spectral resolution from $\mathbb R^n$ into $\{0,u^n_\emptyset\}$. Its characteristic point is $(t^0_1,\ldots,t^0_n)$ and $F_g$ fulfils (v) of Definition \ref{de:n-dim}.
\end{proof}

Using ideas of the present section, we establish one of the main results of the paper.

\begin{theorem}\label{th:main}
Let $M=\Gamma(\mathbb Z\lex G,(1,0))$, where $G$ is a Dedekind $\sigma$-complete $\ell$-group and let $F$ be an $n$-dimensional spectral resolution. Then there is a unique $n$-dimensional observable $x$ on $M$ such that $x((-\infty,t_1)\times\cdots\times (-\infty,t_n))=F(t_1,\ldots,t_n)$, $(t_1,\ldots,t_n)\in \mathbb R^n$.
\end{theorem}

\begin{proof}
We present two different proofs of the statement

\vspace{3mm}
\noindent
{\bf Proof 1.}
\vspace{2mm}

Take $x$ determined by {\rm (\ref{eq:formula})}. Lemma \ref{le:F(x)} says $x$ is an extension of $F$.

\vspace{3mm}
\noindent
{\bf Proof 2.}
\vspace{2mm}
Take the family $\{F_g\colon g \in \{-1,0,1\}^{\{1,\ldots,n\}}\}$ of $n$-dimensional pseudo spectral resolutions on $M$ defined by Lemma \ref{eq:Fg}. Therefore, $F$ can be decomposed as follows
\begin{equation}\label{eq:decomposition}
\begin{split}
    F(t_1,\ldots,t_n)&=\left(\Delta_1(-\infty,\min\{t^0_1,t_1\})+ \Delta_1(\min\{{t^0_1},t_1\},\min\{{t^0_1}^+,t_1\})+ \Delta_1(\min\{{t^0_1}^+,t_1\},t_1)\right)\cdots\\
&\left(\Delta_n(-\infty,\min\{t^0_n,t_n\})+ \Delta_n(\min\{{t^0_n},t_n\},\min\{{t^0_n}^+,t_n\})+ \Delta_n(\min\{{t^0_n}^+,t_n\},t_n)\right)F\\
&=\sum_{g}F_g(t_1,\ldots,t_n).
\end{split}
\end{equation}
Equation (\ref{eq:decomposition}) holds thanks to two points:
(a) For each $i=1,\ldots,n$, we have
$$
\left(\Delta_i(-\infty,\min\{t^0_1,t_1\})+ \Delta_i(\min\{{t^0_1},t_1\},\min\{{t^0_1}^+,t_1\})+ \Delta_i(\min\{{t^0_1}^+,t_1\},t_1)\right)F=\Delta_i(-\infty,t_1)F.
$$
(b) We are using several identities: commutativity~\eqref{eq:II} and distributivity of $\Delta_i$'s operators: For each $i\not=j$ and reals $t\leq t'\leq t''$ and $s\leq s'\leq s''$ we have
\begin{align}
    \big(\Delta_i(t,t')+\Delta_i(t',t'')\big)\Delta_j(s,s')F&= \Delta_i(t,t')\Delta_j(s,s')F+\Delta_i(t',t'')\Delta_j(s,s')F \label{eq:dis1},\\
    \Delta_i(t,t')\big(\Delta_j(s,s')+\Delta_j(s',s'')\big)F&= \Delta_i(t,t')\Delta_j(s,s')F+\Delta_i(t,t')\Delta_j(s',s'')F.\label{eq:dis2}
\end{align}
Equation~\eqref{eq:dis1} holds immediately by definition and the equation~\eqref{eq:dis2} follows from~\eqref{eq:dis1} and the commutativity~\eqref{eq:I}.

Unless the case $g$ equals constantly $0$, $F_g$ is  an $n$-dimensional pseudo spectral resolution with support included in $A_g:=A^{g(1)}_i\times\cdots\times A^{g(n)}_n$, where
$$
A^j_i= \left\{\begin{array}{lll}
(-\infty,t^0_{i})  & \mbox{if}\ j=-1,\\
\{t^0_i\} & \mbox{if} \ j=0,\\
(t^0_{i},\infty)  & \mbox{if}\ j=-1,\\
\end{array}
\right.
$$
and $F_g(\mathbb R^n)\subseteq \mathrm{Rad}(M)$. Let $u_g=F_g(\infty,\ldots,\infty)$. Then $F_g$ is an $n$-dimensional spectral resolution on the interval algebra $[0,u_g]$ which is a $\sigma$-complete MV-algebra.
So we may apply the result from \cite[Thm 5.1]{DvLa3} and extend each $F_g$ with $g\ne 0$ to an $n$-dimensional observable $x_g$ on $[0,u_g]$. In the case of $g$ equals constantly zero, $F_g$ trivially extends to an unique $n$-dimensional observable $x_g$ with support $\{(t^0_1,\ldots,t^0_n)\}$, that is, $x_g(A) = u^n_\emptyset$ if $(t^0_1,\ldots,t^0_n)\in A$, otherwise $x_g(A)=0$, $A \in \mathcal B(\mathbb R^n)$.  Finally, we define
\begin{equation}\label{eq:Fg1}
x=\sum_g x_g.
\end{equation}
From formula~\eqref{eq:decomposition} all the conditions on $x$ to be an observable immediately follow. From the construction of all $F_g$'s, we see that $x((-\infty,t_1)\times \cdots\times (-\infty,t_n))= F(t_1,\ldots,t_n)$, $(t_1,\ldots,t_n)\in \mathbb R^n$.

Using the Sierpi\'nski Theorem, \cite[Thm 1.1]{Kal}, we have the uniqueness of $x$ in both proofs.
\end{proof}

We note that $x$'s defined by (\ref{eq:formula}) or by (\ref{eq:Fg1}) or by (\ref{eq:formula1}) are the same, compare with Theorem \ref{th:TwoEA}.

The following theorem follows the same proof as that of Theorem \ref{th:main}.

\begin{theorem}\label{th:nEA}
If $E=\Gamma_{ea}(\mathbb Z\lex G,(1,0))$ is a perfect effect algebra, where $G$ is a directed monotone $\sigma$-complete po-group with interpolation, then every $n$-dimensional spectral resolution can be extended to a unique $n$-dimensional observable on $E$.
\end{theorem}

\section{$n$-dimensional Spectral Resolutions on $k$-perfect MV-algebras and $k$-perfect Effect Algebras}

Let $M=\Gamma(\mathbb Z \lex G,(k,0))$ and $E=\Gamma_{ea}(\mathbb  Z \lex G,(k,0))$ be a $k$-perfect MV-algebra and a $k$-perfect effect algebra, respectively. We will suppose that $G$ is a Dedekind $\sigma$-complete $\ell$-group in the first case and a directed Dedekind monotone $\sigma$-complete po-group with interpolation in the second case. We say that an $n$-dimensional spectral resolution $F$ has the {\it ordering property} if (1) every non-empty $T_{k_j}$, $k_j=1,\ldots,k$, has at most one characteristic point $\mathbf t_j=(t^{j0}_1,\ldots,t^{j0}_n)$, and (2) if $\mathbf t_i$ and $\mathbf t_j$ are characteristic points of $T_{k_i}$ and $T_{k_j}$ then $\mathbf t_i \le \mathbf t_j$ whenever $0<k_i<k_j\le k$. We remind that $M_0=\{0\}\times G^+$, $M_j= \{j\}\times G$ if $0<j<k$, and $M_k=\{k\}\times G^-$.

First, we establish an analogous result as Claim of Lemma \ref{eq:Fg}. The ordering property gives for each $i=1,\ldots,n$
\begin{equation}\label{eq:chain}
\begin{split}
-\infty <t^{1,0}_i&\le \min\{{t^{1,0}_i}^+,t_i\}\le \min\{t^{2,0}_i,t_i\}\le \min\{{t^{2,0}_i}^+,t_i\}\le\cdots\le \min\{t^{j,0}_i,t_i\}\le \min\{{t^{j,0}_i}^+,t_i\}\\
&\le \min\{t^{j+1,0}_i,t_i\}\le \min\{{t^{j+1,0}_i}^+,t_i\}\le \cdots
\le \min\{t^{l,0}_i,t_i\} \le \min\{{t^{l,0}_i}^+,t_i\}\le t_i.
\end{split}
\end{equation}

The last chain of inequalities suggests the following Lemma.

\begin{lemma}\label{le:claim}
Let $F'$ be an $n$-dimensional pseudo spectral resolution on $M=\Gamma(\mathbb Z\lex G,(k,0))$ with the ordering property. Let $\mathbf t_1\le \cdots\le \mathbf t_l$ be all the characteristic points of $F'$ and let $\mathbf t_j =(t^{0,j}_1,\ldots,t^{0,j}_n)$, $j=1,\ldots,l$. Assume $l>1$.
We define the following functions from $\mathbb R^n$ into $M$
\begin{align*}
F'_{i,1}(t_1,\ldots,t_n)&= \Delta_i(-\infty,\min\{t^{1,0}_i,t_i\})F'(t_1,\ldots,t_n),\\
F'_{i,2j}(t_1,\ldots,t_n)&=\Delta_i(\min\{t^{j,0}_i,t_i\}, \min\{{t^{j,0}_i}^+,t_i\})F'(t_1,\ldots,t_n),\quad j=1,\ldots,l,\\
F'_{i,2j+1}(t_1,\ldots,t_n)&= \Delta_i(\min\{{t^{j,0}_i}^+,t_i\},\min\{t^{j+1,0}_i,t_i\}) F'(t_1,\ldots,t_n),\quad j=1,\ldots,l-1,\\
F'_{i,2l+1}(t_1,\ldots,t_n)&= \Delta_i(\min\{{t^{l,0}_i}^+,t_i\}, t_i)F'(t_1,\ldots,t_n)
\end{align*}
for all $(t_1,\ldots,t_n)\in \mathbb R^n$ and $i=1,\ldots,n$.
Then every $F'_{i,j}$ is an $n$-dimensional pseudo spectral resolution on $M$. In addition, each of $F'_{i,2j}$ has the same characteristic points as $F'$, on the other side, all other functions $F'_{i,j}$ have no characteristic point.
\end{lemma}

\begin{proof}
The volume condition: Let $s_i\le s'_i$ be real numbers for $i=1,\ldots,n$.

(1) Then $\Delta_i(s_i,s'_i)F'_{i,1}= \Delta_i(s_i,s'_i)\Delta_i(-\infty,\min\{t^{1,0}_i,t_i\})F' = \Delta_i(\min\{t^{1,0}_i,s_i\},\min\{t^{1,0}_i,s'_i\})F'$, so that
\begin{align*}
\Delta_1(s_1,s'_1)\cdots\Delta_n(s_n,s'_n)F'_{i,1}&= \Delta_1(s_1,s')\cdots \Delta_{i-1}(s_{i-1},s'_{i-1})
\Delta_i(\min\{t^{1,0}_i,s_i\}, \min\{t^{1,0}_{i},s'_{i}\})\\
&\quad\Delta_{i+1}(s_{i+1},s'_{i+1})\cdots \Delta_n(s_n,s'_n)F'\ge 0.
\end{align*}

(2) We have
$$\Delta_i(s_i,s'_i)F'_{i,2j}=\Delta_i(\min\{t^{j,0}_i,t_i\}, \min\{{t^{j,0}_i}^+,t_i\})F'=\Delta_i(\max\{t^{j,0}_i,t_i\}, \max\{{t^{j,0}_i}^+,t_i\})F',
$$
which yields
\begin{align*}
\Delta_1(s_1,s'_1)\cdots\Delta_n(s_n,s'_n)F'_{i,2j}&=
\Delta_1(s_1,s'_1)\cdots\Delta_{i-1}(s_{i-1},s'_{i-1})
\Delta_i(\max\{t^{j,0}_i,t_i\}, \max\{{t^{j,0}_i}^+,t_i\})\\
&\quad\Delta_{i+1}(s_{i+1},s'_{i+1})\cdots \Delta_n(s_n,s'_n)F'\ge 0.
\end{align*}

(3) There are six cases: (i) $s_i\le t^{j,0}_i\le t^{j+1,0}_i < s'_i$. Then
$\Delta_i(s_i,s'_i)F'_{i,j+1}= \Delta_i({t^{j,0}_i}^+,t^{j+1,0}_i)F$. (ii) $t^{j,0}_i< s_i < s'_i <t^{j+1,0}_i$. Then $\Delta_i(s_i,s'_i)F'_{i,j+1}= \Delta_i(s_i,s'_i)F$. (iii) $t^{j,0}_i \le t^{j+1,0}_i < s_i <s'_i$. Then $\Delta_i(s_i,s'_i)F'_{i,j+1}= 0$. (iv) $t^{j,0}_i <s_i \le t^{j+1,0}_i<s'_i$. Then $\Delta_i(s_i,s'_i)F'_{i,2j+1}= \Delta_i(s_i,t^{j+1,0}_i)F$. (v) $s_i \le t^{j,0}_i < s'_i < t^{j+1,0}_i$. Then $\Delta_i(s_i,s'_i)F'_{i,j+1}= \Delta_i({t^{j,0}_i}^+,s'_i)F$, and (vi) $s_i<s'_i\le t^{j,0}_i\le t^{j+1,0}_i$. Then $\Delta_i(s_i,s'_i)F'_{i,j+1}= 0$. In either case, $\Delta_1(s_1,s'_1)\cdots \Delta_n(s_n,s'_n)F'_{i,j+1}\ge 0$.

(4) There are two cases: (i) $t^{l,0}_i\in [s_i,s'_i)$ which implies
$$
\Delta_i(s_i,s'_i)F'_{i,l+1}=\Delta_i(s_i,s'_i)\Delta_i(\min\{{t^{l,0}_i}, t_i\}, \min\{{t^{l,0}_i}^+,t_i\})F'= \Delta_i(t^{l,0}_i,{t^{l,0}_i}^+)F'.
$$
(ii)
$t^{l,0}_i\notin [s_i,s'_i)$ which implies
$$
\Delta_i(s_i,s_i')F'_{i,2l+1}=\Delta_i(s_i,s'_i)\Delta_i(\min\{{t^{l,0}_i},
t_i\}, \min\{{t^{l,0}_i}^+,t_i\})F'=0.
$$
In either case, we have $\Delta_1(s_1,s'_1)\cdots \Delta_n(s_n,s'_n)F'_{i,2l+1}\ge 0$.

Altogether (1)--(4) entail that every above defined function satisfies the volume condition.

Applying Lemma \ref{lem:help1}, we see that the above defined functions satisfy condition (ii), (iii), (iv)' of Definition \ref{de:n-dim}, so that they are $n$-dimensional pseudo spectral resolutions.

Moreover, every $F'_{i,1}(t_1,\ldots,t_n)\in M_0$, $F'_{i,2j+1}(t_1,\ldots,t_n)\in M_0$, and $F'_{i,2l+1}(t_1,\ldots,t_n)\in M_0$ for each $i=1,\ldots,n$ and all $(t_1,\ldots,t_n)\in \mathbb R^n$.

Suppose that $\mathbf t_j=(t^{0,j},\ldots,t^{0,n})$ is a characteristic point of $F'$ for a fixed $j=1,\ldots,k$. Take $(t_1,\ldots,t_k) \in T_{k_j}=\{(s_1,\ldots,s_n)\in \mathbb R^n\colon F'(s_1,\ldots,s_n)\in M_{k_j}\}$. Then $(t_1,\ldots,t_n)\gg (t^{j,0}_n,\ldots,t^{j,0}_n)$ which gives
$$
F'_{i,2j}(t_1,\ldots,t_n)=F'(t_1,\ldots, {t^{j,0}_i}^+,\ldots,t_n)- F'(t_1,\ldots, t^{j,0}_i,\ldots,t_n)\in M_{k_j-k_{j-1}},
$$
and $k_j-k_{j-1}>0$, moreover, $F_{i,2j}(t^{j,0}_1,\ldots,t^{j,0}_n)=0$ and $F_{i,2j}({t^{j,0}_1}^+,\ldots,{t^{j,0}_n}^+) =  F({t^{j,0}_1}^+,\ldots,{t^{j,0}_n}^+) - F({t^{j,0}_1}^+,\ldots, t^{j,0}_i,\ldots, {t^{j,0}_n}^+)\in  M_{k_j-k_{j-1}}$,
claiming $\mathbf t_j$ is a characteristic point of $F'_{i,2j}$.
\end{proof}

\begin{theorem}\label{th:k-perf}
Let $M=\Gamma(\mathbb Z\lex G,(1,k))$ be a $k$-perfect MV-algebra, where $G$ is a Dedekind $\sigma$-complete $\ell$-group and let $F$ be an $n$-dimensional spectral resolution on $M$ with the ordering property. Then there is a unique $n$-dimensional observable $x$ on $M$ which is an extension of $F$.
\end{theorem}

\begin{proof}
First of all, for each $i=1,\ldots,n$, we set
\begin{align*}
\Delta_i^{1}&= \Delta_i(-\infty,\min\{t^{1,0}_i,t_i\}),\\
\Delta_i^{2j}&= \Delta_i(\min\{t^{j,0}_i,t_i\}, \min\{{t^{j,0}_i}^+,t_i\}),\quad j=1,\ldots,l,\\ \Delta_i^{2j+1}&=\Delta_i(\min\{{t^{j,0}_i}^+,t_i\},
\min\{t^{j+1,0}_i,t_i\}), \quad j=1,\ldots,l-1,\\
\Delta_i^{2l+1}&= \Delta_i(\min\{{t^{l,0}_i}^+,t_i\}, t_i).
\end{align*}
For any $(j_{1},\ldots,j_n) \in \{1,\ldots,2l+1\}^n$, we define
$$
F_{j_1,\ldots,j_n}= \Delta_1^{j_1}\cdots\Delta_n^{j_n}F.
$$
Thanks to Lemma \ref{le:claim}, every $F_{j_1,\ldots,j_n}$ is an $n$-dimensional pseudo spectral resolution such that it has no characteristic point whenever $(j_1,\ldots,j_n)\ne (2k_1,\ldots,2k_n)$, where $k_i= 1,\ldots,l$ for $i=1,\ldots,n$. In such a case, applying \cite[Thm 5.1]{DvLa3}, on the interval $[0,F_{j_1,\ldots,j_n}(\infty,\ldots,\infty)]$, which is a $\sigma$-complete MV-algebra, there is an $n$-dimensional observable $x_{j_1,\ldots,j_n}$ which is an extension of $F_{j_1,\ldots,j_n}$.

On the other hand, due to Lemma \ref{le:claim}, we have that $F_{2k_1,\ldots,2k_n}$ has the same characteristic points as $F$ and it is concentrated in the point $(t^{0,k_1}_1,\ldots,t^{0,k_n})$. Therefore, it is trivial to extend $F_{2k_1,\ldots,2k_n}$ to an $n$-dimensional observable $x_{2k_1,\ldots,2k_n}$. If we define a mapping $x:\mathcal B(\mathbb R^n)\to M$ by
$$
x(A)= \sum \{x_{j_1,\ldots,j_n}(A)\colon (j_1,\ldots,j_n)\in \{1,\ldots,2l+1\}^n\},\quad A \in \mathcal B(\mathbb R^n),
$$
then it is an $n$-dimensional observable on $M$.
If we express $F$ in the form
$$
F(t_1,\ldots,t_n)=\Delta_1(-\infty,t_1)\cdots\Delta_n(-\infty,t_n)F,
$$
and using the chain of inequalities (\ref{eq:chain}) and additivity and multiplication of differences, we get
$$
F(t_1,\ldots,t_n) =\sum_{(j_1,\ldots,j_n)}F_{j_1,
\ldots,j_n}(t_1,\ldots,t_n), \quad (t_1,\ldots,t_n)\in \mathbb R^n,
$$
showing that $x$ is an extension of $F$. The Sierpi\'nski Theorem, \cite[Thm 1.1]{Kal} guarantees the uniqueness of $x$.
\end{proof}

As in Theorem \ref{th:nEA}, we can establish the following Theorem using ideas from Theorem \ref{th:k-perf}.

\begin{theorem}\label{th:nEAk}
If $E=\Gamma_{ea}(\mathbb Z\lex G,(k,0))$ is a $k$-perfect effect algebra, where $G$ is a directed monotone $\sigma$-complete po-group with interpolation, then every $n$-dimensional spectral resolution with the ordering property can be extended to a unique $n$-dimensional observable on $E$.
\end{theorem}

Now, we present a result generalizing Theorem \ref{th:k-perf} and Theorem \ref{th:nEAk} where the increasing property is not more assumed. However, the proof depends in the ideas from Theorem \ref{th:k-perf}.

\begin{theorem}\label{th:k-general}
Let $F$ be an $n$-dimensional spectral resolution on $E=\Gamma_{ea}(H\lex G,(u,0))$, where $G$ is a directed Dedekind  monotone $\sigma$-complete po-group with interpolation and $(H,u)$ is a unital po-group with interpolation. Moreover, suppose $F$ has only finitely many characteristic points. Then $F$ can be extended to an $n$-dimensional observable $x$ on $E$ such that $F=F_x$.
\end{theorem}

\begin{proof}
The proof is rather a direct generalization of the ideas of the second proof of Theorem~\ref{th:main} and Theorem \ref{th:k-perf}. Let $\mathbf t_j=(t^{j,0}_1,\ldots,t^{j,0}_n)$, $j=1,\ldots,k$, be all the characteristic points of $F$. For each $i=1,\ldots,n$, let $m_i=|\{t^{1,0}_i,\ldots,t^{k,0}_i\}|$, that is $m_i$ is the number of mutually different $i$-th coordinates of the characteristic points.
We rewrite and order all the elements of the set $\{t^{1,0}_i,\ldots,t^{k,0}_i\}$ as follows
\begin{equation}\label{eq:order}
t^0_{i,1}<\cdots<t^0_{i,m_i}.
\end{equation}
Then
\begin{align*}
    F(t_1,\ldots,t_n)=&\prod_{i=1}^n \bigg[\Delta_i(-\infty,\min\{t_i,
t_{i,1}^0\})+\Delta_i(\min\{t_i, t_{i,1}^0\},\min\{t_i, t_{i,1}^{0+}\})+\cdots\\
&+ \Delta_i(\min\{t_i, t_{i,1}^{0+}\}, \min\{t_i, t_{i,2}^{0}\})+ \Delta_i(\min\{t_i, t_{i,2}^{0}\}, \min\{t_i, t_{i,2}^{0+}\})+\\
    &\cdots +\Delta_i(\min\{t_i, t_{i,m_i}^0\},\min\{t_i,
t_{i,m_1}^{0+}\})+\Delta_i(\min\{t_i, t_{i,m_i}^{0+}\},t_i)\bigg]F.
\end{align*}

According to Lemma \ref{le:claim}, we can establish the following claim:

\vspace{3mm}
{\it Claim. Let $F'$ be an $n$-dimensional pseudo spectral resolution on $E$. Let $\mathbf t_1,\ldots,\mathbf t_k$ be the system of characteristic points of $F'$. Use $m_i$ defined at the beginning of the proof and the orderings {\rm (\ref{eq:order})} for $i=1,\ldots,n.$

We define the following functions from $\mathbb R^n$ into $E$
\begin{align*}
F'_{i,1}(t_1,\ldots,t_n)&= \Delta_i(-\infty,\min\{t_i,t^0_{i,1}\})F'(t_1,\ldots,t_n),\\
F'_{i,2j}(t_1,\ldots,t_n)&=\Delta_i(\min\{t_i,t^0_{i,j}\}, \min\{t_i,t^{0+}_{i,j}\})F'(t_1,\ldots,t_n),\quad j=1,\ldots,m_i,\\
F'_{i,2j+1}(t_1,\ldots,t_n)&= \Delta_i(\min\{t_i,t^{0+}_{i,j}\},\min\{t_i,t^{0}_{i,j+1}\}) F'(t_1,\ldots,t_n),\quad j=1,\ldots,m_i-1,\\
F'_{i,2m_i+1}(t_1,\ldots,t_n)&= \Delta_i(\min\{t_i,t^{0+}_{i,m_i}\}, t_i)F'(t_1,\ldots,t_n)
\end{align*}
for all $(t_1,\ldots,t_n)\in \mathbb R^n$ and $i=1,\ldots,n$.
Then every $F'_{i,j}$ is an $n$-dimensional pseudo spectral resolution on $E$. In addition, each of $F'_{i,2j}$ has the same characteristic points as $F'$ ($F'_{i,2j}$shares with $F'$ only those characteristic  points, which lie on the  hyperplane $x_i=t^0_{i,j}$), on the other side, all other functions $F'_{i,j}$ have no characteristic point.}
\vspace{2mm}

Now, using the distributivity of $\Delta_i$'s operators, we get $F$ as a sum of $\prod_i(2m_i+1)$ $n$-dimensional pseudo spectral
resolutions that are generated by the claim. Similarly as in the proof of Theorem \ref{th:k-perf}, we can conclude that every summand either has no characteristic point, so due to \cite[Thm 5.1]{DvLa3} it can be extended to an $n$-dimensional observable on an appropriate interval in $E_0$, or it is concentrated in a unique characteristic point. In the second case it can be trivially extended to some $n$-dimensional observable.

As we can extend all the summands to observables, we can extend $F$ to an observable as well.
\end{proof}

We say that an effect algebra $E$ possesses the {\it Observable Existence Property} (OEP, for short) if every $n$-dimensional spectral resolution con $E$ can be extended to an observable and we
denote by $\mathcal{OEP}(EA)$ the class of effect algebras with (OEP).

 The class $\mathcal{OEP}(EA)$ contains these effect algebras (MV-algebras):

\begin{enumerate}
\item[(i)] $\sigma$-complete MV-algebras, \cite[Thm 3.2]{DvKu} for $n=1$.
\item[(ii)] $\sigma$-complete lattice effect algebras, \cite[Thm 3.5]{DvKu} for $n=1$.
\item[(iii)] Boolean $\sigma$-algebras, \cite[Thm 3.6]{DvKu} for $n=1$.
\item[(iv)] $\sigma$-orthocomplete orthomodular posets,
\cite[Thm 3.8]{DvKu} for $n=1$.
\item[(v)] Monotone $\sigma$-complete effect algebras with (RDP),
\cite[Thm 3.9]{DvKu} for $n=1$.
\item[(vi)] $\mathcal E(H)$, \cite[Thm 3.10]{DvKu} for $n=1$.
\item[(vii)] Effect-tribes, \cite[Thm 3.11]{DvKu} for $n=1$.
\item[(viii)] Monotone $\sigma$-complete effect algebras with RIP and DMP, \cite[Thm 4.3]{270} for $n=1$.
\item[(ix)] Every representable monotone $\sigma$-complete effect algebra, \cite[Thm 3.3]{270} for $n=1$.
\item[(x)] Every perfect MV-algebra, \cite[Thm 4.8]{DDL} for $n=1$.

\item[(xi)] Every $n$-perfect MV-algebra, \cite[Thm 3.8]{DvLa}, for $n=1$.

\item[(xii)] Every lexicographic effect algebra, \cite[Thm 3.8]{DvLa1}    for $n=1$.

\item[(xiii)] Every $\sigma$-complete MV-algebra and every Dedekind $\sigma$-complete effect algebra, \cite[Thm 5.1, Thm 5.2]{DvLa3} for general $n\ge 1$.

\item[(xiv)] Every perfect MV-algebra and every perfect effect algebra, Theorems \ref{th:main}--\ref{th:nEA} for $n\ge 1$.

\item[(xv)] Every $k$-perfect MV-algebra and every $k$-perfect effect algebra, Theorem \ref{th:k-general}.

\end{enumerate}

\section{Applications}

The aim of the section is to apply the previous results to show the existence of an $n$-dimensional meet joint observable of $n$ one-dimensional observables and the existence of a sum of $n$-dimensional observables on perfect MV-algebras.

\subsection{$n$-dimensional meet joint observables}

We show that given $n$ one-dimensional observables $x_1,\ldots, x_n$ on an appropriate lexicographic MV-algebra $M$, there is an $n$-dimensional observable $x$ on $M$ such that
\begin{equation}\label{eq:joint}
x((-\infty,s_1)\times\cdots\times (-\infty,s_n))=\bigwedge_{i=1}^n x_i((-\infty,s_i)),\quad s_1,\ldots,s_n \in \mathbb R.
\end{equation}

First we remind that the following forms of distributive laws hold also in lexicographic MV-algebras.

\begin{lemma}\label{le:5.1}
Let $\{x_i\colon i \in I\}$ be a system of elements of an MV-algebra $M$.

{\rm (1)} Let $\bigvee_{i\in I} x_i$ exist in $M$, and let $x$ be any element of $M$. Then $\bigvee_{i\in I}(x\wedge x_i)$ exists in $M$ and
$$
\bigvee_{i\in I}(x\wedge x_i)= x\wedge \bigvee_{i\in I}x_i.
$$

{\rm(2)} If $\bigwedge_{i\in I} x_i$ exists in $M$, then for each $x\in M$, the element $\bigwedge_{i\in I}(x\vee x_i)$ exists in $M$ and
$$
\bigwedge_{i\in I}(x\vee x_i)= x\vee \bigwedge_{i\in I} x_i.
$$
\end{lemma}

\begin{theorem}\label{th:5.2}
Let $x_1,\ldots,x_n$ be one-dimensional observables on a $\Rad$-Dedekind $\sigma$-complete perfect MV-algebra. Then there is a unique $n$-dimensional observable $x$ on $M$ such that {\rm (\ref{eq:joint})} holds.
\end{theorem}

\begin{proof}
Let $F_i(s)=x_i((-\infty,s))$, $s\in \mathbb R$, be a one-dimensional spectral resolution corresponding to the observable $x_i$ for $i=1,\ldots,n$. We are claiming that $F$ is an $n$-dimensional spectral resolution on $M$.

Clearly, $F$, defined by $F(s_1,\ldots,s_n)=\bigwedge_{i=1}^nF(s_i)$, $s_1,\ldots,s_n \in \mathbb R$, is monotone and it satisfies (3.4)--(3.6). We show that it satisfies also the volume condition. Due to \cite[Thm 5.1]{DiLe}, it is sufficient to prove the volume condition for linearly ordered perfect MV-algebra. The following claim was established in \cite[Thm 6.2]{DvLa3}:

\vspace{3mm}
{\it Claim. Let $F_1,\ldots,F_n$, $n\ge 2$, be functions from $\mathbb R$ into a linearly ordered MV-algebra $M$ such that each $F_i$ satisfies the volume condition.  Then $F(s_1,\ldots,s_n)=\bigwedge_{i=1}^n F_i(s_i)$, $s_1,\ldots,s_n\in \mathbb R$, satisfies the volume condition. If, given $A=[a_1,b_1)\times\cdots\times [a_n,b_n)$, we can assume that $F_1(a_1)\le F_2(a_2)\le \cdots\le F_n(a_n)$, then}
\begin{equation}\label{eq:Delta'}
\Delta_1(a_1,b_1)\cdots\Delta_n(a_n,b_n)F(s_1,\ldots,s_n)= \bigwedge_{i=1}^{n-1}(F_i(b_i) \wedge F_n(b_n))- \bigwedge_{i=1}^{n-1}(F_i(b_i)\wedge F_n(a_n)).
\end{equation}
Therefore, $F$ satisfies the volume condition.

Now, we show that $F$ satisfies (v) of Definition \ref{de:n-dim}. For every $i=1,\ldots,n$, let $T^i_1 = \{s\in \mathbb R\colon F_i(s)\in M_1\}$, and let $a^i_1=\bigwedge\{F_i(s)\colon s \in T^i_1\}$. Put $T_1=\{(s_1,\ldots,s_n)\in \mathbb R^n\colon F(s_1,\ldots,s_n)\in M_1\}$.
Then every $a^i_1\in M_1$. Since $F(s_1,\ldots,s_n)\in M_1$ iff $F_i(s_i)\in M_1$ for each $i=1,\ldots,n$, then the element $a=\bigwedge_{i=1}^n a^i_1\in M_1$ and it is a lower bound from $M_1$ for $\{F(s_1,\ldots,s_n)\colon F(s_1,\ldots,s_n)\in M_1\}$. In view of the countability and density of rational numbers and applying Lemma 2.1(2), we see that the element $\bigwedge\{F(s_1,\ldots,s_n)\colon F(s_1,\ldots,s_n)\in M_1\}$ exists in $M$ and it belongs to $M_1$.

Therefore, $F$ is an $n$-dimensional spectral resolution. Due to Theorem \ref{th:main}, there is a unique $n$-dimensional observable $x$ on $M$ such that (\ref{eq:joint}) holds which finishes the proof.
\end{proof}

The observable $x$ from Theorem \ref{th:5.2} is said to be an {\it $n$-dimensional meet joint observable} of $x_1,\ldots, x_n$. We note that using the Sierpi\'nski Theorem, we can show

\begin{equation}\label{eq:joint1}
x(\pi_i^{-1}(A))=x_i(A),\quad A \in \mathcal B(\mathbb R),\ i=1,\ldots,n,
\end{equation}
where $\pi_i:\mathbb R^n \to \mathbb R$ is the $i$-th projection.

In addition, from \eqref{eq:joint1}, we can prove

\begin{equation}\label{eq:joint2}
x(A_1\times\cdots\times A_n)\le \bigwedge_{i=1}^n x_i(A_i),\quad A_1,\ldots, A_n \in \mathcal B(\mathbb R),
\end{equation}
and, in general, it can happen that in \eqref{eq:joint2} we have strict inequality.  Indeed, let $x$ and $y$ be one-dimensional observables on $\Gamma(\mathbb Z\lex \mathbb R,(1,0))$ such that $x(\{2\})=(0,3)$, $x(\{3\})=(1,-3)$, $y(\{1\})=(0,4)$ and $y(\{5\})=(1,-4)$. For the two-dimensional meet joint observable $z$ determined by one-dimensional spectral resolutions $F_x$ and $F_y$, we can show that that in (\ref{eq:joint2}) can be proper inequalities. Otherwise,  we have $(1,0)=(x(\{2\})\wedge y(\{1\}))+ (x(\{2\})\wedge y(\{5\}))+ (x(\{3\})\wedge y(\{1\})) + (x(\{3\})\wedge y(\{5\}))= (0,3)\wedge (0,4) + (0,3)\wedge (1,-4) + (1,-3)\wedge (0,4) + (1,-3)\wedge (1,-4)= (0,3) +(0,3) +(0,4) + (1,-4) = (1,6)$, a contradiction.

\subsection{Sum of $n$-dimensional observables}

Let $(\Omega,\mathcal S)$ be a measurable space and let $f,g:\Omega \to \mathbb R$ be $\mathcal S$-measurable functions. It is well-known that $f+g$ is also $\mathcal S$-measurable. The proof of this fact is based on the simple property
\begin{equation}\label{eq:sum1}
\{\omega \in \Omega: f(\omega)+g(\omega)<t\}= \bigcup_{r \in \mathbb Q}(\{\omega \in \Omega: f(\omega)<r\} \cap \{\omega \in \Omega: g(\omega)<t-r\})
\end{equation}
which holds for each $t \in \mathbb R$, where $\mathbb Q $ is the set of rational numbers, see e.g. \cite[Thm 19.A]{Hal}. This equality was used in \cite[Thm 4.5]{DvLa} to define the sum of two one-dimensional observables on a $k$-perfect MV-algebra $M=\Gamma(\mathbb Z \lex G,(k,0))$, where $G$ is a Dedekind $\sigma$-complete $\ell$-group. More precisely, if $x$ and $y$ are one-dimensional observables on $M$ with spectral resolutions $F_x(t)=x((-\infty,t))$, $F_y(t)=y((-\infty, t))$, $t\in \mathbb R$, we define
$$
F_{x+y}(t):= \bigvee_{r \in \mathbb Q}(F_x(r)\wedge F_y(t-r)),\quad t \in \mathbb R.
$$
Then $F_{x+y}$ is a spectral resolution of an observable $z=x+y$, called the sum of $x$ and $y$. Moreover, $x+y=y+x$.

Similarly, if $T_1,T_2:\Omega \to \mathbb R^n$ are two $n$-dimensional measurable vectors, then $T_1(\omega)=(f_1(\omega),\ldots, f_n(\omega))$ and $T_2(\omega)= (g_1(\omega),\ldots, g_n(\omega))$, $\omega \in \Omega$, for unique $\mathcal S$-measurable functions $f_1,\ldots,f_n,g_1,\ldots,g_n$ on $\Omega$. The sum $T=T_1+T_2$ is $\mathcal S$-measurable thanks to a generalization of (\ref{eq:sum1})
\begin{align*}
\{\omega \colon  T_1(\omega)+T_2(\omega)\in (-\infty,t_1)\times \cdots \times (-\infty,t_n)\}&= \bigcap_{i=1}^n\bigcup_{r_i\in \mathbb Q}\big(\{\omega \colon f_i(\omega)<r_i\}\cap \{\omega\colon g_i(\omega)< t_i-r_i\}\big)\\
&= \bigcap_{i=1}^n \{\omega \colon f_i(\omega)+g_i(\omega)<t_i\},
\end{align*}
which holds for all $t_1,\ldots,t_n\in \mathbb R$. Inspired by this we have the following result.

\begin{theorem}\label{th:sum2}
Let $z_1,z_2$ be $n$-dimensional observables on a perfect MV-algebra $M=\Gamma(\mathbb Z \lex G,(1,0))$, where $G$ is a Dedekind $\sigma$-complete $\ell$-group. Let $\pi_i:\mathbb R^n \to \mathbb R$ be the $i$-th projection, $i=1,\ldots,n$. Define one-dimensional observables $x_i(A)=z_1(\pi^{-1}_i(A))$ and $y_i(A)=z_2(\pi^{-1}_i(A))$ for $A \in \mathcal B(\mathbb R)$ and $i=1,\ldots,n$.
Then there is an $n$-dimensional observable $z=z_1+z_2$ such that
$$
F_{z_1+z_2}(t_1,\ldots,t_n):= \bigwedge_{i=1}^n F_{x_i+y_i}(t_i),\quad t_1,\ldots,t_n \in \mathbb R,
$$
is an $n$-dimensional spectral resolution on $M$ which corresponds to $z$.

If $\mathcal O(M)_n$ is the system of $n$-dimensional observables on $M$, then $\mathcal O(M)_n$ is a commutative semigroup with respect to the binary operation $+$ on $\mathcal O(M)_n$ with the neutral element $o:\mathcal B(\mathbb R^n) \to M$ which is defined by $o(A)=1$ whenever the null vector $(0,\ldots,0)$ belongs to $A$, otherwise $o(A)=0$.
\end{theorem}

\begin{proof}
Due to \cite[Thm 4.5]{DvLa}, every $F_{x_i+y_i}$ is a one-dimensional spectral resolution on $M$ which is the sum of one-dimensional observables $x_i$ and $y_i$ Applying Theorem \ref{th:5.2}, we see that $F_{z_1+z_2}$ is an $n$-dimensional spectral resolution corresponding to an $n$-dimensional meet joint observable of $x_1+y_1,\ldots,x_n+y_n$.

Since $F_{x_i+y_i} = F_{y_i+x_i}$, we see that $F_{z_1+z_2}=F_{z_2+z_1}$ and the operation $+$ is commutative.

We show that $+$ is associative. Let $z_1,z_2,z_3 \in \mathcal O(M)_n$. First, we establish a claim:

\vspace{3mm}
\noindent
{\it Claim.} {\it Let $z_{j,i}(A)=z_j(\pi^{-1}_i(A))$, $(z_1+z_2)_i(A)=(z_1+z_2)(\pi^{-1}_i(A))$, $A \in \mathcal B(\mathbb R)$, $j=1,2$, $i=1,\ldots,n$, then}

$$
(z_1+z_2)_i = z_{1,i}+z_{2,i}, \quad i=1,\ldots,n.
$$
\vspace{2mm}

Indeed, let $i=1,\ldots,n$, be fixed and $t_i\in \mathbb R$. We have
\begin{align*}
F_{(z_1+z_2)_i}(t_i)&=(z_1+z_2)_i((-\infty,t_i))=(z_1+z_2)(\mathbb R\times \cdots\times(-\infty,t_i)\times \cdots \times\mathbb R)\\
&= F_{z_1+z_2}(\infty,\ldots,t_i,\ldots,\infty)= F_{z_{1,i}+z_{2,i}}(t_i),
\end{align*}
which proves the claim.

The associativity of $+$: Let $t_1,\ldots,t_n \in \mathbb R$. Using the claim and associativity of $+$ for one-dimensional observables, we have
\begin{align*}
F_{(z_1+z_2)+z_3}(t_1,\ldots,t_n)&= \bigwedge_{i=1}^n F_{(z_1+z_2)_i+z_{3,i}}(t_i) =\bigwedge_{i=1}^n F_{(z_{1,i}+z_{2,i})+z_{3,i}}(t_i)\\
&= \bigwedge_{i=1}^n F_{z_{1,i}+(z_{2,i}+z_{3,i})}(t_i)
= \bigwedge_{i=1}^n F_{z_{1,i}+(z_2+z_3)_i}(t_i)\\
&=F_{z_1+(z_2+z_3)}(t_1,\ldots,t_n).
\end{align*}
Whence, $(z_1+z_2)+z_3=z_1+(z_2+z_3)$.

The null observable $o$ is evidently a neutral element of the semigroup $\mathcal O(M)_n$.
\end{proof}

\section{Conclusion}

Any measurement of $n$ observables in quantum structures is modeled by an $n$-dimensional observable which is a kind of a $\sigma$-homomorphism from the Borel $\sigma$-algebra $\mathcal B(\mathbb R^n)$ into a quantum structure which is a monotone $\sigma$-complete effect algebra or a $\sigma$-complete MV-algebra. Every observable $x$ restricted to $n$-dimensional infinite intervals of the form $(-\infty,t_1)\times\cdots\times (-\infty,t_n)$, $t_1,\ldots,t_n\in \mathbb R$, defines an $n$-dimensional spectral resolution which is characterized as a mapping from $\mathbb R^n$ into the quantum structure that is monotone, with non-negative increments, and is going to $0$ if one variable goes to $-\infty$ and going to $1$ if all variables go to $+\infty$.

In our case we are concentrated to perfect and $k$-perfect MV-algebras and $k$-perfect effect algebras, and generally to lexicographic quantum structures. Our main ask was to show when we have a one-to-one relationship between $n$-dimensional observables and $n$-dimensional spectral resolutions. In such a case, it was necessary to strengthen the definition of an $n$-dimensional spectral resolution. This model entails a more sophisticated analysis of the problem than in the case of $\sigma$-complete MV-algebras or monotone $\sigma$-complete effect algebras with (RDP), when $k>1$ because in such a case there are appearing more characteristic points than in the case of $n=1$ or in the case $n>1$ and $k>1$.

Therefore, Section 4 is devoted to analysis of characteristic points namely for a two-dimensional case, when it is shown that we have only finitely many characteristic points, see Theorem 4.4. In Theorem 4.6 it was shown that for a general lexicographic MV-algebra a two-dimensional spectral resolution extendable to a two-dimensional observable has to have finitely many characteristic points. On the other hand, every two-dimensional spectral resolution satisfying (3.3)--(3.7) has at most countably many characteristic points, see Theorem 4.8. This is the main content of Part I.

The main body of the paper is in present Part II. We started with  interesting Lemma \ref{le:infty} which shows that elements of the form $F(\infty,\ldots,\infty,t_{j+1},\ldots,t_n)$ are defined in the quantum structure for all $t_{j+1},\ldots,t_n\in \mathbb R$.
In Theorems \ref{th:TwoMV}--\ref{th:TwoEA}, we established that in the case of perfect MV-algebras and perfect effect algebras, there is a one-to-one correspondence between two-dimensional observables and two-dimensional spectral resolutions. The generalization for $n\ge 1$ is given in Theorems \ref{th:main}--\ref{th:k-general}.
For $k>1$ analogous results have been established for observables with the increasing property, see Theorems  \ref{th:k-perf}--\ref{th:nEAk}. Finally, we have applied our results to show that there is an $n$-dimensional meet joint observable of $n$ one-dimensional observables on a perfect MV-algebra, see Theorem \ref{th:5.2}, and we have showed how to define a sum of two $n$-dimensional observables on perfect MV-algebras, Theorem \ref{th:sum2}.

The paper is illustrated by some interesting examples.

We note that still there is open a complete characterization of characteristic points for $n$-dimensional spectral resolutions in $k$-perfect MV-algebras.


\begin{thebibliography}{DvVe2}








\bibitem[DDL]{DDL}
A. Di Nola, A. Dvure\v censkij, G. Lenzi, {\it Observables on perfect MV-algebras}, Fuzzy Sets and Systems {\bf 369} (2019), 57--81. https://doi.org/10.1016/j.fss.2018.11.018


\bibitem[DiLe]{DiLe}
A. Di Nola, A. Lettieri, {\it Perfect MV-algebras are categorically
equivalent to abelian $\ell$-groups}, Studia Logica {\bf 53} (1994),
417--432.



\bibitem[Dvu1]{Dvu1}
A. Dvure\v censkij, {\it  Perfect effect algebras are
categorically equivalent with Abelian interpolation po-groups},
J. Austral. Math. Soc. {\bf 82} (2007), 183--207.

\bibitem[Dvu2]{270}
A. Dvure\v censkij, {\it Representable effect algebras and observables}, Inter. J. Theor. Phys. {\bf 53} (2014),  2855--2866. DOI: 10.1007/s10773-014-2083-z

\bibitem[Dvu3]{304}
A. Dvure\v censkij, {\it Perfect effect algebras and spectral resolutions of observables}, Found. Phys. {\bf 49} (2019), 607--628. DOI: 10.1007/s10701-019-00238-2


\bibitem[DvKu]{DvKu}
A. Dvure\v censkij, M. Kukov\'a,   {\it Observables on quantum structures}, Inf. Sci. {\bf 262} (2014), 215--222. DOI: 10.1016/j.ins.2013.09.014

\bibitem[DvLa]{DvLa}
A. Dvure\v censkij, D. Lachman, {\it Spectral resolutions and observables in $n$-perfect MV-algebras}, Soft Computing {\bf 24} (2020), 843--860. DOI: 10.1007/s00500-019-04543-w

\bibitem[DvLa1]{DvLa1}
A. Dvure\v censkij, D. Lachman, {\it Observables on lexicographic effect algebras}, Algebra Universalis {\bf 80} (2019), Art. 49.
DOI: 10.1007/s00012-019-0628-y

\bibitem[DvLa2]{DvLa2}
A. Dvure\v censkij, D. Lachman, {\it Two-dimensional observables and spectral resolutions}, Rep. Math. Phys. {\bf 85} (2020), 163--191.

\bibitem[DvLa3]{DvLa3}
A. Dvure\v censkij, D. Lachman, {\it Lifting, $n$-dimensional spectral resolutions, and $n$-dimensional observables}, Algebra Universalis  {\bf 34} (2020), Art. Num. 34. DOI: 10.1007/s00012-020-00664-8

\bibitem[DvLa4]{DvLa4}
A. Dvure\v censkij, D. Lachman, {\it $n$-dimensional observables on $k$-perfect MV-algebras and $k$-perfect effect algebras. I. Characteristic point},




\bibitem[Go]{Goo}
K.R. Goodearl, {\it ``Partially Ordered Abelian Groups with Interpolation"},
Math. Surveys and Monographs No. 20, Amer. Math. Soc.,
Providence, Rhode Island, 1986.

\bibitem[Hal]{Hal} 
P.R. Halmos, {\it ``Measure Theory"}, Springer-Verlag, Berlin, 1974.

\bibitem[Kal]{Kal}
O. Kallenberg, {\it ``Foundations of Modern Probability"}, Springer-Verlag, New York, Berlin, Heidelberg, 1997.




\end{thebibliography}
\end{document}